\documentclass[leqno, 12pt]{amsart}
\usepackage{cases}
\usepackage{amsmath}
\usepackage{amsthm}
\usepackage{graphicx,subfigure}
\usepackage{mathrsfs}
\usepackage{bbm}
\usepackage{bm}
\usepackage{amsfonts}
\usepackage{amssymb}
\usepackage{color}
\usepackage{multirow}
\usepackage{enumerate}
\usepackage{epstopdf}
\usepackage{cite,stmaryrd,txfonts}
\usepackage{tikz}
\usepackage{makecell}
\usepackage{booktabs}
\usepackage{hyperref}
\allowdisplaybreaks[4]

\def\opn#1#2{\def#1{\operatorname{#2}}}
\opn\chara{char} \opn\length{\ell}
\opn\projdim{proj\,dim} \opn\injdim{inj\,dim} \opn\rank{rank}
\opn\depth{depth} \opn\grade{grade} \opn\height{height}
\opn\embdim{emb\,dim} \opn\codim{codim}

\opn\Tr{Tr} \opn\bigrank{big\,rank}
\opn\superheight{superheight}\opn\lcm{lcm}
\opn\trdeg{tr\,deg}
\opn\reg{reg} \opn\lreg{lreg}

\opn\Ker{Ker} \opn\Coker{Coker} \opn\Im{Im} \opn\Hom{Hom}
\opn\Tor{Tor} \opn\Ext{Ext} \opn\End{End} \opn\Aut{Aut} \opn\id{id}

\opn\nat{nat}
\opn\pff{pf}
\opn\Pf{Pf} \opn\GL{GL} \opn\SL{SL} \opn\mod{mod} \opn\ord{ord}

\let\to=\rightarrow

\def\Implies{\ifmmode\Longrightarrow \else
     \unskip${}\Longrightarrow{}$\ignorespaces\fi}
\def\implies{\ifmmode\Rightarrow \else
     \unskip${}\Rightarrow{}$\ignorespaces\fi}
\def\iff{\ifmmode\Longleftrightarrow \else
     \unskip${}\Longleftrightarrow{}$\ignorespaces\fi}
\let\gets=\leftarrow

\let\:=\colon

\newtheorem{Theorem}{Theorem}[section]
\newtheorem{Lemma}[Theorem]{Lemma}

\newtheorem{Remark}[Theorem]{Remark}

\newtheorem{Example}[Theorem]{Example}

\newtheorem{Algorithm}[Theorem]{Algorithm}

\newtheorem{Assumption}[Theorem]{Assumption}
\theoremstyle{definition}

\providecommand*{\Dist}[2]{\operatorname{dist}({#1};{#2})}   
\providecommand*{\Dist}[2]{\Dist{#1}{#2}}

\providecommand{\rank}{\operatorname{rank}}                        

\renewcommand{\Im}{\operatorname{Im}}             

\newcommand{\Bc}{{\boldsymbol{c}}}
\newcommand{\Bd}{{\boldsymbol{d}}}
\newcommand{\Be}{{\boldsymbol{e}}}

\newcommand{\Bn}{{\boldsymbol{n}}}

\newcommand{\Bp}{{\boldsymbol{p}}}

\newcommand{\Bv}{{\boldsymbol{v}}}
\newcommand{\Bw}{{\boldsymbol{w}}}
\newcommand{\Bx}{{\boldsymbol{x}}}
\newcommand{\By}{{\boldsymbol{y}}}

\newcommand{\BH}{{\boldsymbol{H}}}

\newcommand{\BL}{{\boldsymbol{L}}}

\newcommand{\BT}{{\boldsymbol{T}}}

\newcommand{\BV}{{\boldsymbol{V}}}
\newcommand{\BW}{{\boldsymbol{W}}}

\newcommand{\BZ}{{\boldsymbol{Z}}}

\newcommand{\phibf}{\boldsymbol{\phi}}

\newcommand{\chibf}{\boldsymbol{\chi}}

\newcommand{\Ce}{\mathcal{E}}

\newcommand{\Cp}{\mathcal{P}}

\newcommand{\Ct}{\mathcal{T}}

\newcommand{\bbA}{\mathbb{A}}

\newcommand{\bbD}{\mathbb{D}}

\newcommand{\bbI}{\mathbb{I}}
\newcommand{\bbJ}{\mathbb{J}}

\newcommand*{\N}[1]{\left\|{#1}\right\|}     
\newcommand*{\SN}[1]{\left|{#1}\right|}      

\newcommand*{\Lp}[2][\defaultdomain]{L^{#2}({#1})}

\newcommand*{\Ltwo}[1][\defaultdomain]{\Lp[#1]{2}}

\newcommand*{\Hm}[2][\defaultdomain]{H^{#2}({#1})}

\newcommand*{\Hone}[1][\defaultdomain]{\Hm[#1]{1}}

\newcommand{\D}{\mathrm{d}}

\newcommand{\ol}{\overline}

\newcommand{\be}{\begin{eqnarray}}
\newcommand{\ee}{\end{eqnarray}}

\newcommand{\ben}{\begin{eqnarray*}}
\newcommand{\een}{\end{eqnarray*}}

\providecommand*{\diag}[1]{\operatorname{diag}\left({#1}\right)}

\opn\ini{in} \opn\inm{inm} \opn\Sym{Sym} \opn\diag{diag}
\opn\Ii{(i)} \opn\Iii{(ii)}

\title{ An unfitted finite element method for PDE-constrained shape optimization via shape gradient flow}
 \author{Wei Gong$^\diamond$, Chuwen Ma$^\dag$ and Ziyi Zhang$^\ddag$}
 \thanks{$^\diamond$SKLMS\&NCMIS, Institute of Computational Mathematics, Academy of Mathematics and Systems Science, Chinese Academy of Sciences, Beijing 100190, China. Email: wgong@lsec.cc.ac.cn. This author was supported by the Strategic Priority Research Program of the Chinese Academy of Sciences XDB0640000 \& XDB0640200, the National Key Research and Development Program of China (2022YFA1004402), and the NSFC under grant no. 12494543 and 12471393.\\
  $^\dag$School of Mathematical Sciences, Key Laboratory of Mathematics and Engineering Applications Ministry of Education,\& Shanghai Key Laboratory of Pure Mathematics and Mathematical Practice, East China Normal University, Shanghai 200241, China.
  Email: cwma@math.ecnu.edu.cn. This author was partially
 supported by NSFC grant No. 12501607, the Science and Technology Commission of
 Shanghai Municipality (No. 22DZ2229014).\\
  $^\ddag$School of Mathematical Sciences, University of Chinese Academy of Sciences, China. Email: zhangziyi21@mails.ucas.ac.cn.}

\begin{document}
\maketitle

{\bf Abstract:}\hspace*{10pt} {In this paper, we propose an unfitted finite element method to solve PDE-constrained shape optimization problems via shape gradient flow. The shape gradient flow system consists of the state equation, the adjoint equation, the velocity equation, as well as the flow map that generates the evolution of the boundary driven by the velocity field, which can be viewed as a limit system of the classical shape gradient descent algorithm. In \cite{GongLiRao} the authors proposed an evolving finite element method to solve the shape gradient flow system. Instead, in this paper, we propose an unfitted finite element method in which the evolution of the boundary is realized by cubic splines and the equations are solved by cut finite element methods with ghost penalization. Under reasonable assumptions, we are able to prove some optimal convergence rates that are further validated by numerical experiments.}

{{\bf Keywords:}\hspace*{10pt} Shape optimization, shape gradient flow, unfitted finite element method, convergence}

\section{Introduction}
\setcounter{equation}{0}
Shape optimization problems have been extensively studied in recent decades due to their wide applications in various fields, including structural mechanics, fluid mechanics, electromagnetics, and biology (cf. \cite{AllaireDapognyJouve}). These models typically look for an optimal domain $\Omega$ by minimizing an objective functional:
\begin{align*}
\min\limits_{\Gamma:=\partial\Omega} \, J(\Gamma)=\int_\Omega j(\Bx,u){\rm d} x,
\end{align*}
where $u$ is the solution to a partial differential equation (PDE) defined in the domain $\Omega$. The study for shape optimization problems ranges from the existence of a solution, the shape sensitivity analysis and the numerical algorithms, we refer to \cite{AllaireDapognyJouve,SokolowskiZolesio1992} for more details.

Shape optimization algorithms involve the evolution of the domain as well as the discretization of the underlined PDEs. For tracking or evolution of the domain, we may resort to the explicit or implicit method such as the moving mesh method, the parametric method, or the level-set method. For the optimization algorithm, shape gradient descent algorithm is the most popular one, where shape calculus plays a pivotal role as it provides the shape gradient by rigorous mathematical theory.  

The convergence issue for shape optimization problems is an important but not fully exploited topic. There are some progresses from the perspective of both numerical PDEs and optimization algorithms. For  convergence involving the discretization of the optimization problem and the governing PDEs, Eppler et al. \cite{EpplerHarbrechtSchneider} considered the boundary parameterization of an elliptic shape optimization problem posed on star-shaped domains and derived error estimates for a finite element method (FEM) using the second-order sufficient optimality condition. In \cite{KinigerVexler} and \cite{FumagalliParoliniVerani} the authors considered  two-dimensional shape optimization problems for the elliptic and Stokes equations, where the portion of the boundary to be optimized is the graph of a function. The second-order convergence of the numerical approximations to a local solution of the optimization problem was proved under the second-order sufficient optimality condition. The analyses in these articles are based on the second-order optimality condition and the computation of the shape Hessian, and are restricted to parametrization of boundaries with special shapes. In \cite{ChenaisZuazua} an abstract convergence of the discrete optimal shape to the continuous optimal shape, measured in the Hausdorff complementary metric, was proved for an elliptic shape optimization problem in two dimensions based on the compactness argument, where the discrete shape is represented by the finite element mesh. For the convergence of the shape gradient descent algorithm, we refer to \cite{DeckelnickHerbertHinze} for the convergence analysis with $\BW^{1,\infty}$ velocity. 

In \cite{GongLiRao} the authors proposed a shape gradient flow system and derived optimal convergence rates for an evolving finite element method. The shape gradient flow consists of the state equation, the adjoint state equation, and the flow map, where the flow map is driven by a velocity field satisfying the following equation with the so-called Hilbertian regularization approach
\begin{equation*}
   \mathrm{Find} \ \Bw \in \BH: \  b(\Bw, \Bv) = -\mathrm{d} J(\Gamma; \Bv) \quad \forall \Bv \in \BH.
\end{equation*}
The choice of the Hilbert space $\BH$ can significantly affect the numerical performance of the shape optimization algorithm, as discussed in \cite{AllaireDapognyJouve}. The most common choice is $\BH:=H^1(\Omega)^d$ which favors  practical implementations and numerical analysis in \cite{GongLiRao}. 

We remark that the shape optimization procedure can be viewed as a moving boundary or interface tracking problem. Conventional approaches to solving these problems can be classified into two main categories: body fitted and unfitted methods. For body fitted method, the moving mesh method \cite{DziukElliott_ESFEM,Edelmann-2021,Elliott-Ranner-2021,Elliott-Styles-2012} is the most popular one, as it offers distinct advantages in precisely tracking domain deformation while maintaining straightforward numerical implementation. However, this method may suffer from progressive mesh quality deterioration, often necessitating computationally expensive remeshing procedures. In contrast, unfitted methods eliminate the need for remeshing by employing alternative interface representation techniques, such as level-set methods \cite{osher1988fronts}, parametric approaches, or cubic spline methods \cite{Ma-Zhang-Zheng-2022} adopted in the current work. The motivation to use cubic splines to approximate the domain boundary characterized by the flow map is that we can add or delete control points during the optimization course, which is in contrast to the parameterization approach. 

The unfitted method for solving interface problems, moving boundary problems, or problems with complex boundaries is quite popular and has been developed very well in recent decades. It decouples the mesh from the physical domain, simplifies the mesh generation, and is especially well-suited for problems with moving boundaries, evolving interfaces, multiphase flows, shape optimization, and so on. Peskin introduced the immersed boundary method (IBM) to simulate flows around complex shapes in \cite{PESKIN1972252}. Then the extended FEM (XFEM) was introduced in \cite{belytschko1999elastic} for fracture modeling with enrichment functions that allow discontinuities across elements.
An improvement of a new technique for modeling cracks in the finite element framework is presented in \cite{moes1999finite}. In \cite{hansbo2002unfitted} the authors laid the groundwork for CutFEM, proposing unfitted discretizations with consistent stabilizations. In \cite{burman2022cutfem} a general framework for the construction and analysis of discrete extension operators was developed with an application to the unfitted finite element approximation of partial differential equations. CutFEM was also used to solve shape optimization problems with satisfactory numerical results, we refer to \cite{Bernland2018,Burman2017,Burman2018} for more details. 

Significant progress has been made in the convergence analysis of finite element approximations for linear parabolic PDEs with moving boundaries using fitted and unfitted elements. In \cite{Ma-Zhang-Zheng-2022} the authors proposed a fourth-order unfitted characteristic FEM to solve the advection-diffusion equation in time-varying domains with a prescribed velocity. However, numerical analyses remain limited for moving boundary problems with solution-driven velocities. In this article, we present an analysis of a moving boundary problem with solution driven velocity, arising from the PDE-constrained shape optimization problem, using the unfitted method. Specifically, we used a cubic spline to track the boundary and defined an artificial mapping between the exact and numerical domains to quantify the geometric error. Furthermore, a careful choice of test functions is introduced to handle the nonlinear terms of the problem. We believe that the numerical analysis provided in \cite{GongLiRao} and the current paper may provide an alternative viewpoint for the convergence analysis of shape optimization problems.

Throughout this paper, let $C$ denote a generic positive constant which may depend on $T$, but is independent of the time step $\tau$, the grid size $h$ and the arc segment $\eta$. Vector-valued quantities
are denoted by boldface symbols, such as $\bm{L}^2(\Omega) = (L^2(\Omega))^2$, and matrix-valued quantities are denoted by blackboard bold symbols, such as $\mathbb{L}^2(\Omega) = (L^2(\Omega))^{2\times 2}$. In the following, $(\cdot,\cdot)_{\Omega}$ denotes the inner products in $L^2(\Omega)$, $\bm{L}^2(\Omega)$, and $\mathbb{L}^2(\Omega)$, in their respective circumstances.

The remainder of this paper is organized as follows. In Section 2, we formulate the shape gradient flow model for PDE-constrained shape optimization problems. Then in Section 3 we give the temporal and spatial discretizations and introduce the boundary tracking algorithm based on the cubic spline. In Section 4, we present and prove our main theoretical result on the convergence of discrete problems solved using an unfitted finite element method. The numerical experiment is shown in Section 5. We conclude the paper in Section 6. 

\section{PDE-constrained shape optimization problems}
\setcounter{equation}{0}
In this paper, we consider the following shape optimization problem 
\begin{align}\label{functional-J}
\min\limits_{\Gamma} \, J(\Gamma)=\int_\Omega j(\cdot,u){\rm d} x
\end{align}
subject to
\begin{align}
\left\{
\begin{aligned}
-\Delta u + u=f\quad&\mbox{in}\ \Omega\subset\mathbb{R}^2,\\
\partial_{\Bn} u=0\quad&\mbox{on}\ \Gamma:=\partial\Omega,
\end{aligned}
\right.
\end{align} 
where $f$ is a given function defined in $\mathbb{R}^2$. The shape density function takes the form of  $j(\cdot,u)=\frac12|\nabla u|^2$ or $j(\cdot,u)=\frac12|u-u_d|^2$, which corresponds to different application background such as minimal energy dissipation or optimal shape reconstruction, respectively. 

For any smooth vector field $\Bv:\mathbb{R}^2\rightarrow \mathbb{R}^2$ and $t \geq 0$ we denote by $\mathcal{F}^t[\Bv]:\Gamma\rightarrow\mathbb{R}^2$ the flow map determined by the velocity field $\Bv$, defined by the following evolution equation 
\begin{equation} \nonumber
\frac{\rm d}{{\rm d} t}\mathcal{F}^t[\Bv]= \Bv\circ \mathcal{F}^t[\Bv]\,\,\,\mbox{on}\,\,\,\Gamma \,\,\,\mbox{with initial condition}\,\,\, \mathcal{F}^0[\Bv]={\rm id}|_{\Gamma} .\label{domain_ODE}
\end{equation} 
The \emph{Eulerian derivative} of $J(\Gamma)$ at $\Gamma$ in the direction $\Bv$ is defined as 
	\begin{equation}\label{def-Eulerian}\nonumber
	{\rm d}J(\Gamma;\Bv):= \frac{{\rm d}}{{\rm d} t} J(\mathcal{F}^t[\Bv](\Gamma) )\Big|_{t=0} = \lim\limits_{t \rightarrow 0^+} \frac{J(\mathcal{F}^t[\Bv](\Gamma) ) - J(\Gamma)}{t}.
	\end{equation}
It can be formulated as an integral on the boundary in terms of the shape gradient defined on the boundary $\Gamma$ (see \cite[Chap. 9, Sec. 3.4]{DelfourZolesio2011}), i.e., 
\begin{equation}
{\rm d}J(\Gamma;\Bv)= \int_\Gamma  J'(\Gamma) \Bv \cdot \Bn {\rm d}\Gamma .\label{Hadamard_structure}
\end{equation}
In fact, $J'(\Gamma)$ is defined as the function on $\Gamma$ satisfying the relation \eqref{Hadamard_structure}; $\Bn$ is the unit outward normal vector of $\Gamma$. 

If we choose $ j (\cdot,u) = \frac{1}{2}|u - u_{d}|^{2} $ defined in \eqref{functional-J}, then the Eulerian derivative of the functional $J(\Gamma)$   has the following closed form (cf. \cite{GongLiZhu,HiptmairPaganiniSargheini2015}): 
\begin{align}\label{distributed_derivative}
    {\rm d} J(\Gamma; \Bv) ={\rm d} J(\Gamma, u,p;\Bv)  := & \int_\Omega 2\nabla u \cdot\bbD(\Bv) \nabla p + p\nabla f \cdot \Bv {\rm d} x \notag\\
					     & + \int_\Omega \Big(\frac{1}{2}|u - u_{d}|^{2} - \nabla u \cdot \nabla p -up + fp\Big) \nabla \cdot \Bv - (u - u_{d})\nabla u_d \cdot \Bv{\rm d} x,
\end{align}
where $\bbD(\Bv) = (\nabla \Bv + (\nabla \Bv)^\top )/2$, $u$ and $p$ are determined by the following equations:
\begin{align*}
\begin{aligned}
-\Delta u + u =f\quad&\mbox{in}\ \Omega,&&\mbox{with}\,\,\, \partial_{\Bn} u=0\,\,\,\mbox{on}\ \Gamma,  \\
-\Delta p + p =u - u_{d} \quad&\mbox{in}\ \Omega,&&\mbox{with}\,\,\,  \partial_{\Bn} p=0\,\,\,\mbox{on}\ \Gamma.
\end{aligned}
\end{align*}
We can also obtain the boundary type shape derivative defined in \eqref{Hadamard_structure} with the shape gradient $J'(\Gamma) = j(\cdot,u)-\nabla u\cdot\nabla p-up+fp$ (cf. \cite{DelfourZolesio2011}).

In general, the shape gradient $J'(\Gamma)$ admits a low regularity, which may cause stability issues in shape optimization algorithms. For the stability and convergence of the numerical approximations, we consider the $H^1$ shape gradient flow of the shape functional in \eqref{functional-J}, i.e., the evolution of the boundary $\Gamma(t)= \partial\Omega(t) $, $t\in[0,T]$, with initial position $\Gamma^0=\partial\Omega^0$, is determined by the following coupled system of equations (cf. \cite{GongLiRao}): 
\begin{subequations}\label{system1}
\begin{align}
-\Delta u + u =f\quad&\mbox{in}\ \Omega(t),&& \partial_{\Bn} u =0\quad\mbox{on}\ \Gamma(t) , \label{eq-u} \\
-\Delta p + p =j'_u(\cdot,u)\quad&\mbox{in}\ \Omega(t),&& \partial_{\Bn} p =0\quad\mbox{on}\ \Gamma(t) , \label{eq-p} \\
-\Delta \Bw+\Bw=0\quad&\mbox{in}\ \Omega(t),&&
\partial_{\Bn} \Bw =- J'(\Gamma(t))\Bn \quad\mbox{on}\ \Gamma(t) ,\label{eq-w} \\
\partial_t\phibf = \Bw\circ\phibf \quad&\mbox{in}\ \Omega^0, && \phibf(\cdot,0)={\rm id}|_{\Omega^0}\quad\mbox{in}\ \Omega^0 , 
\end{align}
\end{subequations}
where $\phibf(\cdot,t):\Omega^0\rightarrow \Omega(t)$ is the flow map that generates the boundary evolution through $\Gamma(t)=\phibf(\Gamma^0,t)$ under the velocity field $\Bw$, $ j'_u $ is the derivative of $ j (\cdot,u) $ with respect to $ u $. Then the rate of change of the shape functional $ J (\Gamma) $ satisfies the following relation: 
\begin{equation}\label{derivative-J}
\frac{{\rm d} J(\Gamma(t))}{{\rm d} t}  = \int_{\Gamma(t)} J'(\Gamma(t)) \Bw(t) \cdot \Bn\, {\rm d}\Gamma(t) .
\end{equation}
Therefore, testing \eqref{eq-w} with $\Bw$ and using the relation \eqref{derivative-J}, the following property can be shown: 
\begin{equation}\nonumber
\frac{{\rm d}J(\Gamma(t))}{{\rm d} t} =-\|\nabla \Bw(t)\|_{L^2(\Omega(t))}^2-\| \Bw(t)\|_{L^2(\Omega(t))}^2 \leq 0 , 
\end{equation} 
i.e., the shape functional decreases as time grows. Consequently, the $H^1$ shape gradient flow evolves to a critical point of the PDE-constrained shape optimization problem.

For any flow map $\phibf:\Omega^0\times[0,T]\rightarrow \mathbb{R}^2$, we denote by $\Omega[\phibf(\cdot,t)]$ the image of $\Omega^0$ under the map $\phibf(\cdot,t)$. If $\Omega[\phibf(\cdot,t)]$ has a Lipschitz boundary, then the boundary type Eulerian derivative in (\ref{Hadamard_structure}) and the distributed type Eulerian derivative in (\ref{distributed_derivative}) are equivalent (cf. \cite{DelfourZolesio2011}). Consequently, by applying integration by parts, \eqref{eq-w} can be equivalently reformulated into the following weak formulation: 
Find $\Bw\in \bm{H}^1(\Omega)$ such that 
\begin{equation}\nonumber
\begin{aligned}
	\int_\Omega \nabla \Bw : \nabla \Bv + \Bw \cdot \Bv \D x 
	=  \int_{\Gamma}\partial_{\Bn}\Bw \cdot \Bv{\rm d} \Gamma 
	&\,=-\int_{\Gamma} J'(\Gamma)  \Bv\cdot\Bn {\rm d} \Gamma\notag\\
	&\,=  -{\rm d} J (\Gamma; \Bv) = -{\rm d}J(\Gamma,u,p; \Bv) \quad \forall\, \Bv \in \bm{H}^1(\Omega) , 
\end{aligned}
\end{equation}
where the closed form of ${\rm d}J(\Gamma,u,p; \Bv)$ is given by \eqref{distributed_derivative}.  Then the moving boundary problem in \eqref{system1} can be written into the following weak formulation: 
\begin{subequations}\label{system2}
\begin{align}
&\int_{\Omega[\phibf(\cdot,t)]} \nabla u\cdot \nabla \chi_u + u\chi_u {\rm d} x = \int_{\Omega[\phibf(\cdot,t)]} f\chi_u {\rm d} x &&\forall\,\chi_u\in H^1(\Omega[\phibf(\cdot,t)]), \label{system2-u}\\
&\int_{\Omega[\phibf(\cdot,t)]} \nabla p\cdot \nabla \chi_p + p\chi_p {\rm d} x = \int_{\Omega[\phibf(\cdot,t)]} j'_u(x,u) \chi_p {\rm d} x &&\forall\,\chi_p\in H^1(\Omega[\phibf(\cdot,t)]), \label{system2-p}\\
&\int_{\Omega[\phibf(\cdot,t)]}  \nabla \Bw : \nabla \chibf_w +\Bw \cdot \chibf_w {\rm d} x 
=- {\rm d} J(\Gamma(t),u,p;\chibf_w) &&\forall\,\chibf_w\in \bm{H}^1(\Omega[\phibf(\cdot,t)]),  \label{system2-w}\\
&\partial_t\phibf = \Bw\circ\phibf, \label{system2-phi}
\end{align}
\end{subequations}
under the initial condition $\phibf(\cdot,0)={\rm id}|_{\Omega^0}$, with $\Gamma(t)=\partial\Omega(t)$.

\section{The unfitted finite element method}\setcounter{equation}{0}
In the following subsections, we present discretizations of the shape gradient flow system in both time and space, as well as the interface tracking approach.
\subsection{Semi-discrete schemes}
Let $0=t_0< t_1< \cdots < t_N=T$ be the uniform partition of $[0,T]$ with step size $\tau=T/N$. The exact characteristic map from $\Gamma(t_m)$ to $\Gamma(t_n)$ is denoted by $\phibf^{m,n}:=\phibf(\cdot,t_n)\circ(\phibf(\cdot,t_m))^{-1}$ for $0\le m\le n\le N$.
The inverse of $\phibf^{m,n}$ is denoted by
\begin{equation}\nonumber
\phibf^{n,m} := \big(\phibf^{m,n}\big)^{-1}.
\end{equation}
Suppose that we are seeking the approximate solution $\phibf^{0,n}_\tau(\Bx)$ to \eqref{system2-phi} and the approximate solution  $u^n_\tau$, $p^n_\tau$, $\Bw^n_\tau$ to \eqref{system2-u}--\eqref{system2-w}, and that $\phibf^{0,m}_\tau(\Bx)$ have been obtained for all $0\le m<n$ and all $\Bx\in\Gamma^0$.
The approximate domain $\Omega_\tau(t_m)$ at $t_m$ is defined in such a way that $\partial \Omega_\tau(t_m) = \Gamma_\tau(t_m)$.
For any $0\le j\le m$, the maps $\phibf^{j,m}_\tau$: $\Omega_\tau(t_j)\to \Omega_\tau(t_m)$ and
$\phibf^{j,m}$: $\Omega(t_j)\to \Omega(t_m)$ are given by 
\begin{equation}\label{Xmn}\nonumber
\phibf_\tau^{j,m} = \phibf_\tau^{0,m}\circ \big(\phibf_\tau^{0,j}\big)^{-1},\qquad
\phibf^{j,m} = \phibf^{0,m}\circ \big(\phibf^{0,j}\big)^{-1}.
\end{equation}
The inverse of $\phibf_\tau^{j,m}$ is denoted by
$\phibf_\tau^{m,j} := \big(\phibf_\tau^{j,m}\big)^{-1}$.

For the time discretization, we consider a linear semi-implicit Euler method:
\begin{equation}\label{eq:DnX0}
\phibf_\tau^{0,n}(\Bx)
= \phibf_\tau^{0,n-1}(\Bx) + \tau\Bw^{n-1}_\tau\circ\phibf_\tau^{0,n-1}(\Bx), \quad \forall\,\Bx\in \Gamma^0.
\end{equation}
By introducing the incremental map $\phibf_\tau^{n-1,n}$ from $t_{n-1}$ to $t_n$, \eqref{eq:DnX0} can be equivalently rewritten as:
\begin{equation}\label{eq:DnX}\nonumber
\phibf_\tau^{n-1,n}(\Bx)
= \Bx + \tau  \Bw^{n-1}_\tau(\Bx) , \quad \forall\,\Bx\in \Gamma_\tau(t_{n-1}).
\end{equation}

Clearly, $\phibf^{n-1,n}_\tau$ defines the approximate domain $\Omega_\tau(t_n)$ in the $n$-${\rm th}$ time step such that $\partial \Omega_\tau(t_n) = \Gamma_\tau(t_n) = \phibf_{\tau}^{n-1,n}(\Bx)$ for all $\Bx\in \Gamma_\tau(t_{n-1})$.
The approximate solutions $u^n_\tau$, $p^n_\tau$ and $\Bw^n_\tau$ satisfy
\begin{subequations}\label{eq:eq wup}
    \begin{align}
    -\Delta u^n_\tau + u^n_\tau =f\quad&\mbox{in}\ \Omega_\tau(t_n),&& \partial_{\Bn} u^n_\tau =0\quad\mbox{on}\ \Gamma_\tau(t_n) , \label{eq-u} \\
-\Delta p^n_\tau + p^n_\tau =j'_u(\cdot,u)\quad&\mbox{in}\ \Omega_\tau(t_n),&& \partial_{\Bn} p^n_\tau =0\quad\mbox{on}\ \Gamma_\tau(t_n) , \label{eq-p} \\
-\Delta \Bw^n_\tau+\Bw^n_\tau=0\quad &\mbox{in}\ \Omega_\tau(t_n),&& \partial_{\Bn} \Bw^n_\tau = -J'(\Gamma_\tau(t_n))\Bn \quad\mbox{on}\ \Gamma_\tau(t_n).
    \end{align}
    \end{subequations}
 Then the moving boundary problem in \eqref{system1} can be written in the following weak formulation: 
\begin{subequations}\label{system3}\nonumber
\begin{align}
&\int_{\Omega_\tau(t_n)} \nabla u_\tau^n\cdot \nabla \chi_u + u_\tau^n \chi_u {\rm d} x = \int_{\Omega_\tau(t_n)} f\chi_u {\rm d} x &&\forall\,\chi_u\in H^1(\Omega_\tau(t_n)), \label{system3-u}\\
&\int_{\Omega_\tau(t_n)} \nabla p_\tau^n\cdot \nabla \chi_p + p_\tau^n \chi_p {\rm d} x = \int_{\Omega_\tau(t_n)} j'_u(x,u_\tau^n) \chi_p {\rm d} x &&\forall\,\chi_p\in H^1(\Omega_\tau(t_n)), \label{system3-p}\\
&\int_{\Omega_\tau(t_n)}  \nabla \Bw_\tau^n : \nabla \chibf_w +\Bw_\tau^n \cdot \chibf_w {\rm d} x 
=- {\rm d} J\left(\Gamma_\tau(t_n),u_\tau^n,p_\tau^n;\chibf_w\right) &&\forall\,\chibf_w\in \bm{H}^1(\Omega_\tau(t_n)),  \label{system3-w}\\
&\phibf_\tau^{0,n+1}(\Bx)
= \phibf_\tau^{0,n}(\Bx) + \tau\Bw^{n}_\tau\circ\phibf_\tau^{0,n}(\Bx), \quad &&\forall\,\Bx\in \Gamma^0, 
\end{align}
\end{subequations}
under the initial condition $\phibf_\tau^0(\cdot)={\rm id}|_{\Gamma^0}$.

\subsection{Interface-tracking and finite element meshes}
In this subsection, we consider spatial discretization. Since we use the unfitted finite element method to solve the coupled system \eqref{system2}, we need to find an efficient way to track the evolution of the domain boundary. In this paper, we assume that the initial domain boundary $\Gamma^0$ can be represented by a cubic spline with control points $\Cp^0=\left\{\Bp^0_j: 0\le j\le J\right\}$. Then the evolution of the domain can be realized by the movement of the control points, which are driven by the discretization of the velocity in \eqref{system2-w}.

We suppose that the arc length of $\Gamma^0$ between $\Bp^0_0$ and $\Bp^0_{j}$ is equal to $L^0_j=j\eta$ for $1\le j\le J$, where $\eta:= L^0/J$ and $L^0$ is the arc length of $\Gamma^0$. For all $0<n \leq N$, suppose that we are given the set of control points $\Bp_j^{n-1}$ at time $t_{n-1}$. Then the control points at $t_n$ are obtained by 
  \begin{equation}
     \Bp_j^n = \Bp_j^{n-1} + \tau  \Bw_h^{n-1}(\Bp_j^{n-1}), \quad 0\le j\le J,\label{control_point}
 \end{equation}
where $\Bw_h^{n}$ will be given later in \eqref{discrete_w}.

During domain evolution, we denote by $\Gamma^n_\eta$ the discrete boundary represented by a cubic spline, while the corresponding approximate domain $\Omega_{\eta}^n$ is defined such that $\partial \Omega_{\eta}^n = \Gamma^n_\eta$. The parametric representation $\BZ^n_{\eta}\in C^2$ of $\Gamma^n_\eta$ is a cubic spline and satisfies
	\begin{equation*}
		\BZ^n_{\eta}(L^n_j) = \Bp^n_j, \quad
		L^n_j = \sum_{i=0}^{j-1}\SN{\Bp^n_{i+1}-\Bp^n_{i}},\quad 0\le j\le J.
	\end{equation*}

Let $D$ be an open bounded domain (e.g. a rectangle) that satisfies $\Omega(t)\subset D$ 
for all $0\le t\le T$.
We denote by $\Ct_h$ a family of uniform partitions of $\bar{D}$ into closed regular squares with the mesh parameter $h$. 
This partition induces two collections of elements that cover $\Omega^n_{\eta}$ and $\Gamma^n_{\eta}$, respectively:  
\begin{equation}
    \begin{aligned}\label{eq:Cth}\nonumber
	   \Ct^n_h &:= \left\{K\in\Ct_h:\; 
		\ol K\cap\Omega^n_{\eta}\neq \emptyset \right\}, \\
		\Ct^n_{h,B} &:= \left\{K\in \Ct^n_h:\; 
		\ol K\cap\Gamma^n_{\eta}\neq \emptyset  \right\} . 
    \end{aligned}
\end{equation}
The covers $\Ct^n_h$, $\Ct_{h,B}^n$ generate  fictitious domains
$\Omega^n_h:= \mathrm{interior}\big(\cup_{K\in\Ct^n_h}K\big)$ and  
$\Omega^n_B:= \mathrm{interior}\big(\cup_{K\in\Ct^n_{h,B}       }K\big)$, respectively. Furthermore, let $\Ce_h$ be the set of all edges in $\Ct_h$. The set of interior edges associated with the boundary elements is defined as:
	\begin{equation}\label{eq:Ceb}\nonumber
		\Ce_{h,B}^{n}= \big\{E\in\Ce_h: \; 
		\exists K\in \Ct^n_{h,B}\;\;
		\hbox{s.t.}\;\; E\subset\partial K
		\backslash \partial \Omega^n_h\big\}.
	\end{equation}
Figure \ref{fig:mesh} provides a visual representation of the meshes and the associated fictitious domains.

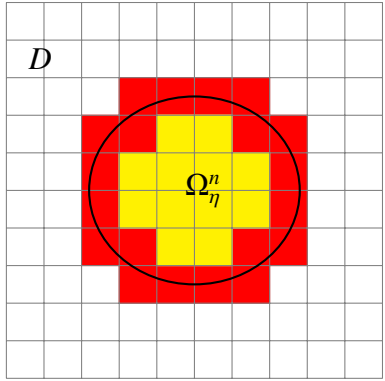
\begin{figure}[!htbp]
\centering
\begin{tikzpicture}[scale =2.5]
        \filldraw[red](0.2*3,0.2*3)--(0.2*3,0.2*2)--(0.2*7,0.2*2)--(0.2*7,0.2*3)--(0.2*8,0.2*3)--(0.2*8,0.2*7)--(0.2*7,0.2*7)--(0.2*7,0.2*8)--(0.2*3,0.2*8)--(0.2*3,0.2*7)--(0.2*2,0.2*7)--(0.2*2,0.2*3)--(0.2*3,0.2*3);
        
        \filldraw[white](0.2*3,0.2*6)--(0.2*4,0.2*6)--(0.2*4,0.2*7)--(0.2*5,0.2*7)--(0.2*6,0.2*7)--(0.2*6,0.2*6)
        --(0.2*7,0.2*6)--(0.2*7,0.2*5)--(0.2*7,0.2*4)--(0.2*6,0.2*4)--(0.2*6,0.2*3)--(0.2*5,0.2*3)--(0.2*4,0.2*3)
        --(0.2*4,0.2*4)--(0.2*3,0.2*4)--(0.2*3,0.2*5)--(0.2*3,0.2*6);

        \filldraw[yellow](0.2*3,0.2*6)--(0.2*4,0.2*6)--(0.2*4,0.2*7)--(0.2*5,0.2*7)--(0.2*6,0.2*7)--(0.2*6,0.2*6)
         --(0.2*7,0.2*6)--(0.2*7,0.2*5)--(0.2*7,0.2*4)--(0.2*6,0.2*4)--(0.2*6,0.2*3)--(0.2*5,0.2*3)--(0.2*4,0.2*3)--(0.2*4,0.2*4)--(0.2*3,0.2*4)--(0.2*3,0.2*5)--(0.2*3,0.2*6);

         \draw [step =0.2cm,gray,thin] (0,0) grid (2cm,2cm);
         \draw[black, thick] (1,1) ellipse [x radius=0.56cm, y radius=0.5cm];
    \node[left] at (0.3,1.7) {$D$};
    \node[left] at (1.2,0.2*5) {$\Omega_\eta^n$};
\end{tikzpicture}
\caption{\small  The domain $D$ and its partition $\Ct_h$, the approximate domain $\Omega_\eta^n$, the set of red and yellow squares $\Ct^n_h$, the set of red squares $\Ct^n_{h,B}$, the union of red and yellow squares $\Omega^n_h$, the union of red  squares $\Omega_B^n$.} \label{fig:mesh}
\end{figure}

Next, we define in the following the finite element spaces in $D$, $\Omega^n_h$ and $\Omega_B^n$, respectively,
\begin{align*}
V(k,\Ct_h) :=\,& \big\{v\in\Hone[D]:
		v|_K\in Q_k(K),\;\forall\,K\in \Ct_h\big\}, \\
V(k,\Ct^n_h):=\,&  \big\{v|_{\Omega^n_h}:
		v\in V(k,\Ct_h) \big\} , \quad 
V(k,\Ct^n_{h,B}):=\,\big\{v|_{\Omega_B^n}:
		v\in V(k,\Ct_h) \big\},
\end{align*}
where $Q_k$ contains polynomials with each variable having degree at most $k$. 
The space of piecewise regular functions over $\Ct^n_h$ is defined by
\begin{equation}\label{eq:FEM space}\nonumber
H^m(\Ct^n_h) := \big\{v\in\Ltwo[\Omega^n_h]:\;
v|_K\in H^m(K),\;\forall\, K\in\Ct^n_h\big\},\qquad m\ge 1.
\end{equation}

\subsection{Fully discrete schemes}
In this subsection, we present the fully discrete schemes for the shape gradient flow system. 
Once we obtain $\Gamma_{\eta}^n$, we use the standard unfitted finite element methods to discretize \eqref{eq:eq wup}. That is, we look for $u_h^n\in V(k,\Ct_{h}^n)$,  $p_h^n \in V(k,\Ct_{h}^n)$ and $\Bw_h^n \in \BV(k,\Ct_h^n) := V(k,\Ct_{h}^n)^2$  such that 
\begin{subequations}\label{fe_weak}
\begin{align}
     & \int_{\Omega_{\eta}^n} \nabla u_h^n \cdot \nabla \chi_u  + u_h^n \chi_u \D x + \mathscr{J}_h^n(u_h^n,\chi_u) = \int_{\Omega_{\eta}^n} f \chi_u \mathrm{d} x\quad &&\forall \chi_u\in V(k,\Ct_{h}^n),\label{discrete_u}\\
    & \int_{\Omega_{\eta}^n} \nabla p_h^n \cdot \nabla \chi_p + p_h^n \chi_p\mathrm{d} x + \mathscr{J}_h^n(p_h^n,\chi_p) = \int_{\Omega_{\eta}^n} j_u^{\prime}\left(x, u_h^n\right) \chi_p \mathrm{d} x\quad &&\forall \chi_p\in V(k,\Ct_{h}^n), \label{discrete_p}\\
    & \int_{\Omega_{\eta}^n} \nabla \Bw_h^n: \nabla \chibf_w+\Bw_h^n \cdot \chibf_w \mathrm{d} x + \mathscr{J}_h^n(\Bw_h^n,\chibf_w) =-\mathrm{d} J\left( \Gamma_{\eta}^n, u_h^n, p_h^n ; \chibf_w\right)\quad &&\forall \chibf_w\in \BV(k,\Ct_{h}^n),\label{discrete_w}
\end{align}
\end{subequations}
where $\mathscr{J}_h^n$ is the ghost-penalty stabilization term\cite{hansbo2018cutfem}, defined by
\begin{align}\label{eq:ghost penalty}
    \mathscr{J}_h^n(u,v) &:= \alpha \sum_{E\in \Ce_{h,B}^n} \sum_{s=1}^k h^{2s-1}\int_E \llbracket{\partial_{\bm{n}}^s u}\rrbracket \llbracket{\partial_{\bm{n}}^s v}\rrbracket \D s.
\end{align}
Here $\alpha$ is a positive constant whose value will be specified in the section on numerical experiments, $\partial^s_{\bm{n}} u$ denotes the $s^{\text{th}}$-order normal derivative of $u$ on the edge $E$, and $\bm{n}$ denotes the unit outward normal of $E$. Moreover, the jump of a quantity $v$ over a face $E$ is defined as $\llbracket{v}\rrbracket =v|_{K_1}-v|_{K_2}$, with $E = K_1 \cap K_2 \in \Ce_{h,B}^n$. 
Now we can define a new bilinear form on $H^1(\Omega_h^n)$ as follows:
\begin{align}\nonumber
    \mathscr{A}_h^n(u,v):= (\nabla u,\nabla v)_{\Omega_{\eta}^n} + (u, v)_{\Omega_{\eta}^n} + \mathscr{J}_h^n(u,v),
\end{align}
then the discrete form can be written as 
\begin{subequations}\label{fully_discrete}
\begin{align}
    &\mathscr{A}_h^n(u_h^n,\chi_u) = (f,\chi_u)_{\Omega_{\eta}^n} \quad &&\forall \chi_u\in V(k,\Ct_{h}^n), \\
    &\mathscr{A}_h^n(p_h^n,\chi_p) = (j_u'(\cdot,u_h^n),\chi_p)_{\Omega_{\eta}^n} \quad &&\forall \chi_p\in V(k,\Ct_{h}^n), \\
    &\mathscr{A}_h^n(\Bw_h^n,\chibf_w)= -\D J(\Gamma_{\eta}^n,u_h^n,p_h^n;\chibf_w) \quad &&\forall \chibf_w\in \bm{V}(k,\Ct_{h}^n).
\end{align}
\end{subequations}

\begin{Assumption}\cite[Assumption 4.1]{Ma-Zhang-Zheng-2022}\label{norm-equi}
There exist an integer $I>0$ and a constant $\gamma >0$ which are independent of $h$ and $\tau$ such that, for any $K\in \mathcal{T}^n_{h,B}$, one can find at most $I$ elements $\{K_j \}_{j=1}^I \subset \mathcal{T}^n_{h}$ satisfying $K_1 = K, K_{j-1} \cap K_j \in \Ce_{h,B}^{n}$ for $1 < j \leq I$ and that $K_I \cap \Omega_{\eta}^n$ contains a disk of radius $\gamma h$.
\end{Assumption}

Under Assumption \ref{norm-equi}, we know from \cite[Lemma 4.2]{Ma-Zhang-Zheng-2022} that the penalty term $\mathscr{J}_h^n(\cdot,\cdot)$ has the following properties for any $v_h \in V(k,\Ct^n_h)$:
\begin{subequations}\label{eq:norm inequality}
\begin{align}
    |v_h|^2_{H^1(\Omega_h^n)} &\leq C\left(\mathscr{J}_h^n(v_h, v_h) + |v_h|^2_{H^1(\Omega_{\eta}^n)}\right),\label{eq:norm inequality-h1} \\
    ||v_h||^2_{L^2(\Omega_h^n)} &\leq C\left(h^2\mathscr{J}_h^n(v_h, v_h) + ||v_h||^2_{L^2(\Omega_{\eta}^n)}\right),\label{eq:norm inequality-l2} 
\end{align}
\end{subequations}
Moreover, one can show that 
\begin{eqnarray}\label{coercivity}
   C\|v_h\|^2_{H^1(\Omega_h^n)}\leq \mathscr{A}_h^n(v_h,v_h),\quad \mathscr{A}_h^n(u_h,v_h)\leq C\|u_h\|_{H^1(\Omega_h^n)}\|v_h\|_{H^1(\Omega_h^n)}\quad\forall u_h,v_h\in V(k,\Ct^n_h).
\end{eqnarray}
Then the state equation, the adjoint state equation, and the velocity equation in (\ref{fully_discrete}) admit a unique solution, respectively.

For the error analysis of the domain approximation, we need to define the discrete flow map function $\phibf_h^n$. Similarly to the definition of $\phibf_\tau^{0, n}$, we define  $\phibf_h^n$ as
\begin{equation}\label{dis_map}
\phibf_h^n(\Bx)
= \phibf_h^{n-1}(\Bx) + \tau\Bw^{n-1}_h\circ\phibf_h^{n-1}(\Bx) \quad \forall\,\Bx\in \Gamma^0,
\end{equation}
with the initial condition $\phibf_h^0(\cdot) = {\rm id}|_{\Gamma^0}$. Furthermore, we can extend the above definition to the interior of $\Omega^0$, such that
\begin{equation}\label{dis_map_domain}
\phibf_h^n(\Bx)
= \phibf_h^{n-1}(\Bx) + \tau\Bw^{n-1}_h\circ\phibf_h^{n-1}(\Bx) \quad \forall\,\Bx\in \Omega^0.
\end{equation}
We denote by $\Omega^n$ the image of $\Omega^0$ under the mapping $\phibf_h^n$, i.e., $\Omega^n:=\phibf_h^n(\Omega^0)$. We note that $\partial\Omega^n$ is continuous and consists of piecewise segments. According to \eqref{control_point} and \eqref{dis_map_domain}, $\Omega^n$ coincide with $\Omega_\eta^n$ at the control points $\{\Bp_j^n\}$. Although they are not identical, the latter can be viewed as a cubic spline interpolation of the former, or the former can be viewed as a piecewise interpolation of the latter.

In the following analysis, we extend the finite element solution $\Bw_h^{n-1}$ from $\Omega_h^{n-1}$ to $D$ such that $\Bw_h^{n-1}$ vanishes at the nodes outside the domain $\bar{\Omega}_h^{n-1}$. Then we have $||\Bw_h^{n-1}||_{\BH^1(D)} \leq C ||\Bw_h^{n-1}||_{\BH^1( \Omega_h^{n-1})}$ (cf. \cite{Ma-Zhang-Zheng-2022,HuangWuXiao17}). After extension of $\Bw_h^{n-1}$ to $D$, one can guarantee that the mapped point $\phibf_h^{n-1}(\Bx) \in \Omega^{n-1} \subset D$ for any $\Bx \in \Omega^0$, which means that $\Bw_h^{n-1}$ is well-defined at the point $\phibf_h^{n-1}(\Bx)$. Then the extension of the definition in (\ref{dis_map}) from $\Gamma^0$ to $\Omega^0$ as in (\ref{dis_map_domain}) is meaningful. We should emphasize here that the definition of the discrete flow map is compatible with the evolution of control points, since for any $0 \leq j \leq J$ and $0 < n \leq N$ we have $\Bp_j^{n}=\phibf_h^n(\Bp_j^{n-1})$ according to \eqref{control_point}. In actual computations, we use the control points \eqref{control_point} for interface tracking to reduce computational cost. The domain extension \eqref{dis_map_domain} is introduced only for theoretical analysis, since it allows us to work on a fixed reference domain when deriving error estimates.

\section{Error Analysis}\setcounter{equation}{0}

We can define Sobolev extensions of the exact solutions $u^n:=u(t_n)$, $p^n:=p(t_n)$, $\Bw^n:=\Bw(t_n)$ from $\Omega(t_n)$ to $\mathbb{R}^2$ which remain invariant in $\Omega(t_n)$, and still denote them as $u^n$, $p^n$, $\Bw^n$. Let $\phibf^n$ be the exact flow map of time $t_n$ and $\phibf_h^n \in V(k,\Ct^0_h)$ be the discrete solution defined in \eqref{dis_map_domain}. Let $\hat{u}^n, \hat{p}^n\in V(k,\Ct^n_h)$ and $\hat{\Bw}^n \in \bm{V}(k,\Ct^n_h)$ be the Lagrange interpolations of the Sobolev extensions $u^n, p^n, \Bw^n$, while $u_h^n, p_h^n, \Bw_h^n$ are the finite element solutions. For any $\chi_u,\chi_p\in V(k,\Ct^n_h)$ and $\chibf_w\in \bm{V}(k,\Ct^n_h)$, we have the following weak formulations:
\begin{subequations}\label{interpolation_weak}
\begin{align}
& \int_{\Omega_{\eta}^n} \nabla \hat{u}^n \cdot \nabla \chi_u + \hat{u}^n \chi_u \mathrm{d} x+ \mathscr{J}_h^n(\hat{u}^n,\chi_u) = \int_{\Omega_{\eta}^n} f \chi_u \mathrm{d} x+\int_{\Omega_{\eta}^n} d_u^n \chi_u \mathrm{d} x, \label{eq:hat u}\\
& \int_{\Omega_{\eta}^n} \nabla \hat{p}^n \cdot \nabla \chi_p + \hat{p}^n \chi_p\mathrm{d} x + \mathscr{J}_h^n(\hat{p}^n,\chi_p) = \int_{\Omega_{\eta}^n} j_u^{\prime}\left(x, \hat{u}^n\right) \chi_p \mathrm{d} x + \int_{\Omega_{\eta}^n} d_p^n \chi_p \mathrm{d} x, \label{eq:hat p} \\
& \int_{\Omega_{\eta}^n} \nabla \hat{\Bw}^n: \nabla \chibf_w+\hat{\Bw}^n \cdot \chibf_w \mathrm{d} x + \mathscr{J}_h^n(\hat{\Bw}^n,\chibf_w) =-\mathrm{d} J\left( \Gamma_{\eta}^n, \hat{u}^n, \hat{p}^n ; \chibf_w\right) + \int_{\Omega_{\eta}^n} \Bd_w^n \cdot \chibf_w \mathrm{~d} x, \label{eq:hat w}
\end{align}
\end{subequations}
where $d_u^n$, $d_p^n$ and $\Bd_w^n$ are some defects that will be estimated in subsection \ref{consistency}.
Here, $\Omega(t_n)$ denotes the true domain and $\Omega^n_{\eta}$ denotes the cubic spline domain of the FE solution. Define
\begin{equation}\label{eq:sta err}
\begin{aligned}
    &\Be_\phi^n(\Bx) = \phibf_h^n(\Bx) - \phibf^n(\Bx),  &&\Be_w^n(\Bx) = \Bw_h^n(\Bx) - \hat{\Bw}^n(\Bx), \\
    &e_{u}^n(\Bx) = u_h^n(\Bx) - \hat{u}^n(\Bx), &&e_{p}^n(\Bx) = p_h^n(\Bx) - \hat{p}^n(\Bx),
\end{aligned}
\end{equation}
then the following estimate holds.
\begin{Theorem}\label{Thm:4.1}
    Suppose that for all $t \in [0, T]$, the domain $\Omega(t)$ is of class $C^2$. Furthermore, assume that the given function $f \in C^k$, the flow map $\phi(\cdot, t) \in C^2([0,T];C^1(\Omega^0))$, the exact solutions $u(\cdot, t), p(\cdot, t)$ and the exact velocity $\Bw(\cdot, t)$ are of type  $C([0,T];C^k(D))$. Then for  $k \geq 2$, $\tau=\mathcal{O}(h^k)$, $\eta=\mathcal{O}(h^{k+\frac{1}{2}})$, and sufficiently small $h$, the following estimate holds for any $0 \leq n \leq N$:
    \begin{equation}\label{error}
        \N{\Be_\phi^n}_{H^1(\Omega^0)} + \N{e_u^n}_{H^1(\Omega_h^n)} +\N{e_p^n}_{H^1(\Omega_h^n)} +\N{\Be_w^n}_{\BH^1(\Omega_h^n)} \leq  C(\tau + h^k ).
    \end{equation}
\end{Theorem}
\begin{Remark}
We remark that the above error estimate is only local in time during the shape optimization procedure. For our error analysis, we need the assumption that the arc length between two neighboring control points should be uniformly bounded from above and below by the parameter $\eta$, which, however, may be violated during the evolution of the domain boundary. In such cases, we can add or remove control points as done in \cite{Ma-Zhang-Zheng-2022}. We also refer to \cite{GongLiRao} for an evolving finite element method where remeshing may be necessary for large shape deformations. 
\end{Remark}

To prove Theorem~\ref{Thm:4.1}, we require the following assumption to hold  for all time steps $0 \leq n \leq N$:
\begin{equation}\label{assum}
    \|\Be_\phi^n\|_{W^{1,\infty}(\Omega^0)} \leq \min\left\{1,\frac{1}{2 C_\phi}\right\}, \quad
    ||e_u^n||_{W^{1,\infty}(\Omega_h^n)} \leq 1,\quad
    ||e_p^n||_{W^{1,\infty}(\Omega_h^n)} \leq 1,\quad
    \|\Be_w^n\|_{\BW^{1,\infty}(\Omega_h^n)} \leq 1,
\end{equation}
where $C_\phi := \|(\nabla \phibf)^{-1}\|_{L^{\infty}(0,T;L^\infty(D))}$.
Together with the smoothness of the exact solution and \eqref{assum}, the triangle inequality yields the bounds:
 \begin{subequations}
    \begin{align}
    \|\phibf_h^n\|_{W^{1,\infty}(\Omega^0)} &\leq\, \|\Be_\phi^n\|_{W^{1,\infty}(\Omega^0)} + \|\phibf^n\|_{W^{1,\infty}(\Omega^0)}\leq 1+\|\phibf\|_{L^\infty(0,T;W^{1,\infty}(\Omega^0))}, \label{eq:bd of phi_h}\\
    \|u_h^n\|_{W^{1,\infty}(\Omega_h^n)} & \leq \, \|e_u^n\|_{W^{1,\infty}(\Omega_h^n)}
    +\|\hat{u}^n\|_{W^{1,\infty}(\Omega_h^n)}\nonumber\\
    &\leq \|e_u^n\|_{W^{1,\infty}(\Omega_h^n)}
    +\|\hat{u}^n-u^n\|_{W^{1,\infty}(\Omega_h^n)} + \|u^n\|_{W^{1,\infty}(\Omega_h^n)} \label{eq:bd of uh}\\
    &\leq 1 + Ch + \|u\|_{L^\infty(0,T;W^{1,\infty}(D))} \leq 2 + \|u\|_{L^\infty(0,T;W^{1,\infty}(D))}, \nonumber\\
    \|p_h^n\|_{W^{1,\infty}(\Omega_h^n)} & \leq \, \|e_p^n\|_{W^{1,\infty}(\Omega_h^n)}
    +\|\hat{p}^n\|_{W^{1,\infty}(\Omega_h^n)} \leq 2 + \|p\|_{L^\infty(0,T;W^{1,\infty}(D))}, \label{eq:bd of ph}\\
    \|\Bw_h^n\|_{\BW^{1,\infty}(\Omega_h^n)} & \leq \, \|\Be_w^n\|_{\BW^{1,\infty}(\Omega_h^n)}
    +\|\hat{\Bw}^n\|_{\BW^{1,\infty}(\Omega_h^n)} \leq 2 + \|\Bw\|_{L^\infty(0,T;\BW^{1,\infty}(D))}. \label{eq:bd of wh}
    \end{align}
 \end{subequations}
We establish~\eqref{assum} by mathematical inductions and briefly outline the proof strategy as follows. Suppose that for some $0 \leq n_0 < N$, the estimates in~\eqref{assum} hold for all $0 \leq n \leq n_0$. We then demonstrate that the error bound~\eqref{error} holds at time level $n = n_0 + 1$, which in turn implies that~\eqref{assum} is valid for $n = n_0 + 1$. Consequently, if~\eqref{assum} is satisfied for the initial step $n = 0$, the desired result follows for all $n\leq N$ by induction.

The complete inductive argument is presented in detail in the following. First, if $n=0$, we know that both $\phibf_h^0$ and $\phibf^0$ are identity functions, so that $||\Be_{\phi}^0||_{W^{1,\infty}(\Omega^0)} = 0$. Meanwhile, using (\ref{coercivity}) we know that
\begin{equation}\label{estimate_0}
    C \|e_u^0\|_{H^1(\Omega_h^0)}^2 \leq \mathscr{A}_h^0(e_u^0, e_u^0) = \mathscr{A}_h^0(e_u^0, u_h^0 - u^0 + u^0 - \hat{u}^0) =  \mathscr{A}_h^0(e_u^0, u^0 - \hat{u}^0) + \mathscr{A}_h^0(e_u^0, u_h^0 - u^0).
\end{equation}
Since $u$ is sufficiently smooth, we have $\mathscr{J}_h^0(u^0, e_u^0) = 0$. Given that $\Omega^0$ is prescribed, the domains of integration in \eqref{fe_weak} and \eqref{system2} coincide. By subtracting \eqref{fe_weak} from \eqref{system2} and choosing $e_u^0$ as the test function, the term $\mathscr{A}_h^0(e_u^0, u_h^0 - u^0)$ vanishes. Now, (\ref{estimate_0}) becomes
\begin{equation}\nonumber
    C\|e_u^0\|_{H^1(\Omega_h^0)}^2 \leq \mathscr{A}_h^0(e_u^0, u^0 - \hat{u}^0) \leq \|e_u^0\|_{H^1(\Omega_h^0)} \|u^0 - \hat{u}^0\|_{H^1(\Omega_h^0)} + C h^k \|e_u^0\|_{H^1(\Omega_h^0)}, 
\end{equation}
which yields
\begin{equation}\nonumber
    \|e_u^0\|_{W^{1,\infty}(\Omega_h^0)} \leq Ch^{-1}\|e_u^0\|_{H^1(\Omega_h^0)} \leq C h^{-1} \|u^0 - \hat{u}^0\|_{H^1(\Omega_h^0)} + Ch^{k-1}.
\end{equation}
When we take $k \geq 2$ and $h$ small enough, we have
\begin{equation}\nonumber
    \|e_u^0\|_{W^{1,\infty}(\Omega_h^0)} \leq Ch^{k-1} \leq C h \leq 1.
\end{equation}
Therefore, $e_p^0$ and $\Be_w^0$ can be estimated in a similar way.

Then for any $ 1 \leq m \leq N$, we assume that the property \eqref{assum} holds for all $0 \leq n \leq m-1$. Next, we consider the case $n = m$ and show that, under this circumstance, the error estimate \eqref{error} holds for $n = m$.

\subsection{Estimate for ${\bf e}_\phi^n$ under the property \eqref{assum}}

\begin{Lemma}\label{flow_inver}
     If the time step size $\tau$ is sufficiently small, then the discrete flow map $\phibf_h^n $ defined in (\ref{dis_map_domain}) is invertible for all $0 \leq n \leq m $ with a uniform bound independent of $n$.
 \end{Lemma}
 \begin{proof}
By definition, for any $0 \le n \le m$ we have
\[
\nabla \phibf_h^n
= \bigl(I+\tau \nabla \Bw_h^{\,n-1}\circ \phibf_h^{\,n-1}\bigr)\,\nabla \phibf_h^{\,n-1}
= \prod_{j=0}^{n-1}\bigl(I+\tau \nabla \Bw_h^{\,j}\circ \phibf_h^{\,j}\bigr).
\]
Let $C_w := \|\Bw\|_{L^\infty(0,T;W^{1,\infty}(D))}$ and assume that the time step $\tau$ is sufficiently small such that $(2+C_w)\tau < 1/4$.
Then according to \eqref{eq:bd of phi_h} and \eqref{eq:bd of wh}, $\phibf_h^n:\Omega^0 \to \Omega^n$ is bijective and its Jacobian matrix $\nabla \phibf_h^n(x)$ is invertible for all $x\in\Omega^0$ and $0\le n\le m$.
Moreover, the inverse Jacobian is uniformly bounded:
\[
\|(\nabla \phibf_h^n)^{-1}\|_{L^\infty(\Omega^n)}
\le \left(\frac{1}{1-(2+C_w)\tau}\right)^{n}.
\]
Since $n \le N = T/\tau$, it follows that
\[
\lim_{\tau\to 0^+}\left(\frac{1}{1-(2+C_w)\tau}\right)^{T/\tau} = e^{(2+C_w)T},
\]
hence, $\|(\nabla \phibf_h^n)^{-1}\|_{L^\infty(\Omega^n)}$ remains uniformly bounded independent of $n$ for $0\le n\le m$.
\end{proof}

\begin{Lemma}\label{lem: ephi}
Under the assumptions of Theorem~\ref{Thm:4.1} and the property \eqref{assum}, for any
$1\le n\le m$ it holds that
\[
\|\Be_\phi^n\|_{H^1(\Omega^0)}
\le C\left(\tau
+ \tau\sum_{l=1}^n \|\Be_w^{l-1}\|_{\BH^1(\Omega_h^{l-1})}
+ \tau\sum_{l=1}^n \|\Be_\phi^{l-1}\|_{H^1(\Omega^0)}
+ h^k\right),
\]
where $C>0$ is independent of $\tau$ and $h$, but may depend on $T$ and the norms of the exact solution.
\end{Lemma}

\begin{proof}
By the definition of $\phibf_h^n$ and a Taylor expansion of the exact flow map,
for all $\Bx\in\Omega^0$,
\begin{align*}
\phibf_h^n(\Bx) &= \phibf_h^{n-1}(\Bx) + \tau\,\Bw_h^{n-1}\circ\phibf_h^{n-1}(\Bx),\\
\phibf^{n}(\Bx) &= \phibf^{n-1}(\Bx) + \tau\,\Bw^{n-1}\circ\phibf^{n-1}(\Bx) + \mathcal{O}(\tau^2).
\end{align*}
Subtracting the two relations and using $\Be_\phi^n=\phibf_h^n-\phibf^n$ yields
\[
\|\Be_\phi^n\|_{L^2(\Omega^0)}
\le \|\Be_\phi^{n-1}\|_{L^2(\Omega^0)}  + C\tau^2
+ \tau \|\Bw_h^{n-1}\circ\phibf_h^{n-1}-\Bw^{n-1}\circ\phibf^{\,n-1}\|_{\BL^2(\Omega^0)}.
\]
We divide the last term by inserting $\hat{\Bw}^{\,n-1}$:
\begin{align*}
&\|\Bw_h^{n-1}\circ \phibf_h^{n-1}-\Bw^{n-1}\circ \phibf^{\,n-1}\|_{\BL^2(\Omega^0)} \\
\le\;&\|\Be_w^{n-1}\circ \phibf_h^{n-1}\|_{\BL^2(\Omega^0)}
+ \|(\hat{\Bw}^{\,n-1}-\Bw^{n-1})\circ \phibf_h^{n-1}\|_{\BL^2(\Omega^0)}
+ C\|\Bw^{n-1}\|_{\BW^{1,\infty}(D)}\|\Be_\phi^{n-1}\|_{L^2(\Omega^0)}.
\end{align*}
By a change of variables $\Bx=\phibf_h^{n-1}(\Bx^0)$, $\Bx^0 \in \Omega^0$ and Lemma~\ref{flow_inver},
\[
\|\Be_w^{n-1}\circ \phibf_h^{n-1}\|_{\BL^2(\Omega^0)} \le C\|\Be_w^{n-1}\|_{\BL^2(\Omega_h^{n-1})},
\quad
\|(\hat{\Bw}^{\,n-1}-\Bw^{n-1})\circ \phibf_h^{n-1}\|_{\BL^2(\Omega^0)}
\le C\|\hat{\Bw}^{\,n-1}-\Bw^{n-1}\|_{\BL^2(\Omega_h^{n-1})}.
\]
Using the interpolation error estimate, we obtain $\|\hat{\Bw}^{\,n-1}-\Bw^{n-1}\|_{\BL^2(\Omega_h^{n-1})}\le Ch^{k+1}$, and thus
\[
\|\Be_\phi^n\|_{L^2(\Omega^0)}
\le \|\Be_\phi^{n-1}\|_{L^2(\Omega^0)} + C\tau^2
+ C\tau \|\Be_w^{n-1}\|_{\BL^2(\Omega_h^{n-1})}
+ C\tau \|\Be_\phi^{n-1}\|_{L^2(\Omega^0)}
+ C\tau h^{k+1}.
\]

Next, differentiating and applying the chain rule gives
\[
|\Be_\phi^n|_{H^1(\Omega^0)}
\le |\Be_\phi^{n-1}|_{H^1(\Omega^0)} + C\tau^2
+ \tau \|\nabla(\Bw_h^{n-1}\circ\phibf_h^{n-1})-\nabla(\Bw^{n-1}\circ\phibf^{\,n-1})\|_{\mathbb{L}^2(\Omega^0)}.
\]
We estimate the gradient term by inserting $\hat{\Bw}^{\,n-1}$ and applying the chain rule:
\begin{align*}
&\|\nabla(\Bw_h^{n-1}\circ\phibf_h^{n-1})
-\nabla(\Bw^{n-1}\circ\phibf^{\,n-1})\|_{\mathbb{L}^2(\Omega^0)} \\
=& \|((\nabla\Bw_h^{n-1})\circ\phibf_h^{n-1})\,\nabla\phibf_h^{n-1}
-((\nabla\Bw^{n-1})\circ\phibf^{\,n-1})\,\nabla\phibf^{\,n-1}\|_{\mathbb{L}^2(\Omega^0)} \\
\le&
\|((\nabla\Bw_h^{n-1}-\nabla\hat{\Bw}^{\,n-1})\circ\phibf_h^{n-1})\,\nabla\phibf_h^{n-1}\|_{\mathbb{L}^2(\Omega^0)} \\
&+
\|((\nabla\hat{\Bw}^{\,n-1}-\nabla\Bw^{n-1})\circ\phibf_h^{n-1})\,\nabla\phibf_h^{n-1}\|_{\mathbb{L}^2(\Omega^0)} \\
&+
\|((\nabla\Bw^{n-1})\circ\phibf_h^{n-1})\,\nabla\phibf_h^{n-1}
-((\nabla\Bw^{n-1})\circ\phibf^{\,n-1})\,\nabla\phibf^{\,n-1}\|_{\mathbb{L}^2(\Omega^0)} .
\end{align*}
Using the boundedness of $\nabla\phibf_h^{n-1}$ from Lemma~\ref{flow_inver},
the change of variables $\Bx=\phibf_h^{n-1}(\Bx^0)$, $\Bx^0 \in \Omega^0$,
and the Lipschitz continuity of $\nabla\Bw^{n-1}$, we obtain
\begin{align*}
&\|\nabla(\Bw_h^{n-1}\circ\phibf_h^{n-1})
-\nabla(\Bw^{n-1}\circ\phibf^{\,n-1})\|_{\mathbb{L}^2(\Omega^0)} \\
\le\;&
C\|\nabla\Be_w^{n-1}\circ\phibf_h^{n-1}\|_{\mathbb{L}^2(\Omega^0)}
+C\|\nabla(\hat{\Bw}^{\,n-1}-\Bw^{n-1})\circ\phibf_h^{n-1}\|_{\mathbb{L}^2(\Omega^0)}
+C\|\Bw^{n-1}\|_{\BW^{2,\infty}(D)}\|\Be_\phi^{n-1}\|_{H^1(\Omega^0)} \\
\le\;&
C\|\Be_w^{n-1}\|_{\BH^1(\Omega_h^{n-1})}
+ C h^{k}
+ C\|\Be_\phi^{n-1}\|_{H^1(\Omega^0)}.
\end{align*}
Combining the above bounds yields, for $1\le n\le m$,
\[
\|\Be_\phi^n\|_{H^1(\Omega^0)}
\le \|\Be_\phi^{n-1}\|_{H^1(\Omega^0)} + C\tau^2
+ C\tau \|\Be_w^{n-1}\|_{\BH^1(\Omega_h^{n-1})}
+ C\tau \|\Be_\phi^{n-1}\|_{H^1(\Omega^0)}
+ C\tau h^{k}.
\]
Summing from $l=1$ to $n$ and using $n\tau\le T$ gives
\[
\|\Be_\phi^n\|_{H^1(\Omega^0)}
\le C\tau
+ C\tau\sum_{l=1}^n \|\Be_w^{l-1}\|_{\BH^1(\Omega_h^{l-1})}
+ C\tau\sum_{l=1}^n \|\Be_\phi^{l-1}\|_{H^1(\Omega^0)}
+ C h^{k},
\]
which completes the proof.
\end{proof}

\subsection{Construction and estimation of the auxiliary map}
To estimate the defects induced by the mismatch between the discrete and exact domains, we introduce an auxiliary map that measures the distance between the two domains
\[
\Omega^n := \phibf_h^n(\Omega^0)
\quad\text{and}\quad
\Omega(t_n)=\phibf^n(\Omega^0).
\]
By Lemma~\ref{flow_inver}, the following definition is well posed
\begin{equation}\nonumber
\Be_v^n(\Bx) := \phibf^n\circ(\phibf_h^n)^{-1}(\Bx)-\Bx,
\qquad \forall\,\Bx\in\Omega^n.
\end{equation}
A key identity follows immediately from \eqref{eq:sta err}:
for any $\Bx^0\in\Omega^0$,
\begin{equation}\label{ev_ephi_identity}
\Be_v^n\big(\phibf_h^n(\Bx^0)\big)
= \phibf^n(\Bx^0)-\phibf_h^n(\Bx^0)
= -\,\Be_\phi^n(\Bx^0),
\qquad\text{hence}\qquad
\Be_v^n = -\,\Be_\phi^n\circ(\phibf_h^n)^{-1}\ \ \text{in }\Omega^n.
\end{equation}

\begin{Lemma}\label{lem:vtophi}
Under the assumptions of Theorem~\ref{Thm:4.1} and the property \eqref{assum},
for any $0\le n\le m-1$ it holds that
\begin{equation}\label{vtophi_H1}
\|\Be_v^n\|_{H^1(\Omega_h^n)} \le C\,\|\Be_\phi^n\|_{H^1(\Omega^0)},
\end{equation}
where $C>0$ is independent of $h,\tau$, and $n$.
\end{Lemma}

\begin{proof}
By \eqref{ev_ephi_identity} and the change of variables $\Bx=\phibf_h^n(\Bx^0)$,
\begin{align*}
\|\Be_v^n\|_{L^2(\Omega^n)}^2
&=\int_{\Omega^n}|\Be_v^n(\Bx)|^2\,\D x
=\int_{\Omega^0}\big|\Be_v^n(\phibf_h^n(\Bx^0))\big|^2\,|\det\nabla\phibf_h^n(\Bx^0)|\,\D x \\
&=\int_{\Omega^0}|\Be_\phi^n(\Bx^0)|^2\,|\det\nabla\phibf_h^n(\Bx^0)|\,\D x
\le C\,\|\Be_\phi^n\|_{L^2(\Omega^0)}^2,
\end{align*}
where we used the uniform bound on $|\det\nabla\phibf_h^n|$ from Lemma~\ref{flow_inver}.
Moreover, differentiating \eqref{ev_ephi_identity} gives, for $\Bx=\phibf_h^n(\Bx^0)$,
\[
\nabla\Be_v^n(\Bx)\,\nabla\phibf_h^n(\Bx^0) = -\,\nabla\Be_\phi^n(\Bx^0),
\qquad\text{thus}\qquad
\nabla\Be_v^n(\Bx) = -\,\nabla\Be_\phi^n(\Bx^0)\,(\nabla\phibf_h^n(\Bx^0))^{-1}.
\]
Hence, using again the change of variables and Lemma~\ref{flow_inver},
\begin{align*}
|\Be_v^n|_{H^1(\Omega^n)}^2
&=\int_{\Omega^n}|\nabla\Be_v^n(\Bx)|^2\,\D x \\
&=\int_{\Omega^0}\Big|\nabla\Be_\phi^n(\Bx^0)\,(\nabla\phibf_h^n(\Bx^0))^{-1}\Big|^2\,|\det\nabla\phibf_h^n(\Bx^0)|\,\D x
\le C\,|\Be_\phi^n|_{H^1(\Omega^0)}^2.
\end{align*}
Combining the estimates under $L^2$ norm and $H^1$ semi-norm yields \eqref{vtophi_H1}.
\end{proof}

With the auxiliary map $\Be_v^n$, we define a one-parameter family of transformations
\begin{equation}\label{def:Ttheta}
\BT_\theta^n(\Bx):=\Bx+\theta\,\Be_v^n(\Bx),
\qquad \Bx\in\Omega^n,\ \ \theta\in[0,1],
\end{equation}
and the associated deformed domains
\[
\Omega_\theta^n:=\BT_\theta^n(\Omega^n),
\qquad \Omega_0^n=\Omega^n,\qquad \Omega_1^n=\Omega(t_n).
\]

\begin{Lemma}\label{lem:est T_theta}
Under the assumptions of Theorem~\ref{Thm:4.1} and property \eqref{assum},
the mapping $\BT_\theta^n:\Omega^n\to\Omega_\theta^n$ is invertible for every $0\le n\le m-1$ and $\theta\in[0,1]$.
Moreover, $\BT_\theta^n$ is bi-Lipschitz on $\Omega^n$, and its Jacobian matrix
$\bbJ_\theta^n:=\nabla\BT_\theta^n$ exists a.e.\ in $\Omega^n$ with the uniform bounds
\[
\|\bbJ_\theta^n\|_{L^\infty(\Omega^n)}+\|(\bbJ_\theta^n)^{-1}\|_{L^\infty(\Omega^n)}\le C,
\]
where $C>0$ is independent of $n$ and $\theta$.

\end{Lemma}

\begin{proof}
Fix $\Bx\in\Omega^n$ and set $\Bx^0:=(\phibf_h^n)^{-1}(\Bx)\in\Omega^0$.
By \eqref{def:Ttheta} and \eqref{ev_ephi_identity}, we have
\[
\BT_\theta^n(\Bx)=\Bx-\theta\,\Be_\phi^n(\Bx^0),
\qquad
\bbJ_\theta^n(\Bx)=\nabla\BT_\theta^n(\Bx)
= \bbI -\theta\,\nabla\Be_\phi^n(\Bx^0)\,(\nabla\phibf_h^n(\Bx^0))^{-1}
\quad\text{a.e.\ in }\Omega^n.
\]
Using $\nabla\phibf_h^n=\nabla\phibf^n+\nabla\Be_\phi^n$, we factorize
\[
\nabla\phibf_h^n=(\bbI +\bbA_{\theta})\nabla\phibf^n,
\qquad
\bbA_{\theta}:=\nabla\Be_\phi^n(\nabla\phibf^n)^{-1}.
\]
Hence, whenever $\bbI+\bbA_{\theta}$ is invertible,
\[
(\nabla\phibf_h^n)^{-1}=(\nabla\phibf^n)^{-1}(\bbI+\bbA_{\theta})^{-1}.
\]
By \eqref{assum} and the definition of $C_\phi$, $\|\bbA_{\theta}\|_{L^\infty(\Omega^0)}\le \frac12$,
hence $(\bbI+\bbA_{\theta})^{-1}$ exists and $\|(\bbI+\bbA_{\theta})^{-1}\|_{L^\infty(\Omega^0)}\le 2$. Therefore,
\[
\|\nabla\Be_\phi^n(\nabla\phibf_h^n)^{-1}\|_{L^\infty(\Omega^0)}
=\|\bbA_{\theta}(\bbI+\bbA_{\theta})^{-1}\|_{L^\infty(\Omega^0)}
\le \frac{\|\bbA_{\theta}\|_{L^\infty(\Omega^0)}}{1-\|\bbA_{\theta}\|_{L^\infty(\Omega^0)}}\le 1.
\]
Consequently, for $\theta\in[0,1)$,
\[
\|\bbJ_\theta^n-\bbI\|_{L^\infty(\Omega^n)}
=\theta\,\|\nabla\Be_\phi^n(\nabla\phibf_h^n)^{-1}\|_{L^\infty(\Omega^0)}
\le \theta<1,
\]
which implies that $\bbJ_\theta^n$ is invertible a.e.\ and
$\|(\bbJ_\theta^n)^{-1}\|_{L^\infty(\Omega^n)}\le (1-\theta)^{-1}$.
Moreover, for any $\Bx_1,\Bx_2\in\Omega^n$,
\[
(1-\theta)|\Bx_1-\Bx_2|\le |\BT_\theta^n(\Bx_1)-\BT_\theta^n(\Bx_2)|
\le (1+\theta)|\Bx_1-\Bx_2|,
\]
so $\BT_\theta^n$ is bi-Lipschitz and hence globally invertible.
For $\theta=1$, we have $\BT_1^n=\phibf^n\circ(\phibf_h^n)^{-1}$, which is bi-Lipschitz according to 
Lemma~\ref{flow_inver} and the regularity of $\phibf^n$. The stated uniform bounds follow.
\end{proof}

For any $\By\in\Omega_\theta^n$, let $\Bx=(\BT_\theta^n)^{-1}(\By)$ and define the associated velocity field
\begin{equation}\label{eq:DT_theta}
\Be_{v,\theta}^n(\By):=\frac{\D}{\D\theta}\BT_\theta^n(\Bx)
=\Be_v^n(\Bx)
=\Be_v^n\circ(\BT_\theta^n)^{-1}(\By),
\qquad \By\in\Omega_\theta^n.
\end{equation}
By the change of variables and Lemma~\ref{lem:est T_theta}, we obtain
\begin{equation}
\begin{aligned} \label{eq:ev_H1}
\|\Be_{v,\theta}^n\|_{L^2(\Omega_\theta^n)}^2
&=\int_{\Omega^n}|\Be_v^n(\Bx)|^2\,|\det(\bbJ_{\theta}^n(\Bx))|\,\D x
\le C\,\|\Be_v^n\|_{L^2(\Omega^n)}^2, \\ 
|\Be_{v,\theta}^n|_{H^1(\Omega_\theta^n)}^2
&=\int_{\Omega^n}\big|\nabla\Be_v^n(\Bx)\,(\bbJ_\theta^n(\Bx))^{-1}\big|^2\,
|\det(\bbJ_\theta^n(\Bx))|\,\D x
\le C\,|\Be_v^n|_{H^1(\Omega^n)}^2. 
\end{aligned}
\end{equation}

\subsection{Consistency estimate}\label{consistency}
To define $\hat{u}^n_{\theta} \in H^1(\Omega_{\theta}^n)$, let $\hat{u}^n_0:= \hat{u}^n$, i.e., the interpolation of the exact solution $u^n$, and define
\begin{equation*}
    \hat{u}^n_{\theta} \left(\By \right) := u_h^n(\By) - \big (u_h^n\circ (\BT_{\theta}^n)^{-1}(\By) - \hat{u}_0^n\circ (\BT_{\theta}^n)^{-1}(\By)\big)=u_h^n(\By)-\big(u_h^n(\Bx)-\hat{u}_0^n(\Bx)\big) \quad \By \in \Omega_{\theta}^n.
\end{equation*}
The error between $u_h^n$ and $\hat{u}_{\theta}^n$ is given by 
\begin{equation*}
    e_{u, \theta}^n(\By): =  u_h^n(\By) - \hat{u}^n_{\theta}(\By) = u_h^n(\Bx) - \hat{u}_0^n(\Bx) = e_{u, 0}^n(\Bx)=e_u^n(\Bx) \quad \By \in \Omega_{\theta}^n,
\end{equation*}
implying its material derivative vanishes. Specifically, at $\By=T_{\theta}^n(\Bx) \in \Omega_{\theta}^n$, we have
\begin{equation*}
     \partial^{\bullet} e^n_{u,\theta}(\By) = \lim_{\alpha\to 0^{+}}
     \frac{e_{u,\theta}^n(\Bx + (\theta+\alpha)\Be_v^n(\Bx)) - e_{u,\theta}^n(\Bx+\theta \Be_v^n(\Bx))}{\alpha} =0.
\end{equation*}
Moreover, the material derivative of the gradient satisfies
\begin{equation}\label{grad_material}
    \partial^{\bullet} \nabla e_{u, \theta}^n = \nabla \partial^{\bullet}  e_{u, \theta}^n - (\nabla \Be_{v, \theta}^n)^{\top} \nabla e_{u, \theta}^n = - (\nabla \Be_{v, \theta}^n)^{\top} \nabla e_{u, \theta}^n.
\end{equation}

\subsubsection{Estimate for $d_u^n$}\label{sec:estimate for du}
Recall that $e_{u, 0}^n(\Bx)=e_u^n(\Bx) \in V(k,\Ct_h^n)$ and $e_{u, \theta}^n(\By) \in H^1(\Omega_{\theta}^n)$. The weak formulation of the exact solution is to find $u^n\in H^1(\Omega(t_n))$ such that 
\begin{equation}
\begin{aligned}\label{4.12}
    \int_{\Omega(t_n)} \nabla u^n \cdot \nabla \chi_u +u^n \chi_u \mathrm{d} x = \int_{\Omega(t_n)} f \chi_u \mathrm{d} x\quad \forall\chi_u\in H^1(\Omega(t_n)).
\end{aligned}
\end{equation}
Assuming that $u^n$ is smooth enough, we have $\mathscr{J}_h^n(u^n,e_{u,0}^n) = 0$. Since $e_{u,1}^n\in H^1(\Omega(t_n))$, taking $\chi_u = e_{u,1}^n$ in \eqref{4.12}, it follows from \eqref{eq:hat u} that
\begin{equation}\label{du}
\begin{aligned}
    \int_{\Omega_{\eta}^n} d_u^n e_{u,0}^n \mathrm{d}x 
    &= \int_{\Omega_{\eta}^n} \left( \nabla \hat{u}^n \cdot \nabla e_{u,0}^n + \hat{u}^n e_{u,0}^n - f e_{u,0}^n \right) \mathrm{d}x + \mathscr{J}_h^n(\hat{u}^n, e_{u,0}^n) \\
    &= \underbrace{\bigg( \int_{\Omega_{\eta}^n} \nabla \hat{u}^n \cdot \nabla e_{u,0}^n \mathrm{d}x - \int_{\Omega(t_n)} \nabla u^n \cdot \nabla e_{u,1}^n \mathrm{d}x \bigg)}_{\text{I}^{u}} 
   + \underbrace{\bigg( \int_{\Omega_{\eta}^n} \hat{u}^n e_{u,0}^n \mathrm{d}x - \int_{\Omega(t_n)} u^n e_{u,1}^n \mathrm{d}x \bigg)}_{\text{II}^u} \\
    &\quad + \underbrace{\bigg( \int_{\Omega(t_n)} f e_{u,1}^n \mathrm{d}x - \int_{\Omega_{\eta}^n} f e_{u,0}^n \mathrm{d}x \bigg)}_{\text{III}^u} 
    \quad + \underbrace{\bigg( \mathscr{J}_h^n(\hat{u}^n, e_{u,0}^n) - \mathscr{J}_h^n(u^n, e_{u,0}^n) \bigg)}_{\text{IV}^u}.
\end{aligned}
\end{equation}

\paragraph{\textbf{Estimate of $\mathrm{I}^u$.}}
We decompose
\begin{equation}\label{eq:Iu_split}
\begin{aligned}
\mathrm{I}^u
&=\underbrace{\Big(\int_{\Omega_\eta^n}\nabla \hat u^n\cdot \nabla e_{u,0}^n\,\D x
-\int_{\Omega^n}\nabla \hat u^n\cdot \nabla e_{u,0}^n\,\D x\Big)}_{\mathrm{I}_1^u}
+\underbrace{\int_{\Omega^n}\nabla(\hat u^n-u^n)\cdot \nabla e_{u,0}^n\,\D x}_{\mathrm{I}_2^u}\\
&\quad+\underbrace{\Big(\int_{\Omega^n}\nabla u^n\cdot \nabla e_{u,0}^n\,\D x
-\int_{\Omega(t_n)}\nabla u^n\cdot \nabla e_{u,1}^n\,\D x\Big)}_{\mathrm{I}_3^u}.
\end{aligned}
\end{equation}
By the interpolation estimate, $\|\nabla(\hat u^n-u^n)\|_{L^2(\Omega^n)}\le Ch^k$, hence
\begin{equation}\label{eq:I2u_bd}\nonumber
|\mathrm{I}_2^u|\le Ch^k\,\|e_{u,0}^n\|_{H^1(\Omega_h^n)}.
\end{equation}

To bound $\mathrm{I}_1^u$, we use the following geometric strip estimate.

\begin{Lemma}\label{Lemma:4.4}
For any $v_h\in V(k,\Ct_h^n)$ and $v\in H^1(D)$, there holds
\begin{equation}\label{eq:strip_est}
\|v_h\|_{L^2(\Omega_\eta^n\setminus\Omega^n)}^2+\|v_h\|_{L^2(\Omega^n\setminus\Omega_\eta^n)}^2
\le C h^{-1}\eta\,\|v_h\|_{L^2(\Omega_h^n)}^2,
\qquad
\|v\|_{L^2(\Omega_\eta^n\setminus\Omega^n)}^2+\|v\|_{L^2(\Omega^n\setminus\Omega_\eta^n)}^2
\le C\eta\,\|v\|_{H^1(D)}^2.
\end{equation}
\end{Lemma}

\begin{proof}
From the definition in \eqref{dis_map_domain}, it follows that $\partial \Omega^n = \phibf_h^n(\Gamma^0)$ 
coincides with the control points $\{\Bp_j^n\}$ of $\Omega_{\eta}^n$, where $\Omega^n$ and $\Omega_{\eta}^n$ intersect. 
Since $\partial \Omega_{\eta}^n$ can be viewed as a cubic spline interpolation of $\partial \Omega^n$, the standard estimate implies that
$$\text{dist}(\phibf_h^n(\Bx), \partial \Omega_{\eta}^n)\leq C \eta \|\phibf_h^n\|_{W^{1,\infty}(\Gamma^0)} \quad \forall \Bx \in \Gamma^0,$$
where $C>0$ is independent of $\eta$.
 Using the condition (\ref{assum}), by the inverse estimate and the extension of finite element function defined before, we have
\begin{equation}\nonumber
    \N{v_h}_{L^2(\Omega_{\eta}^n \backslash \Omega^n)}^2 \leq C \sum_{K \cap (\Omega_{\eta}^n \backslash \Omega^n) \neq \emptyset} h \eta \N{v_h}_{L^{\infty}(K)}^2  \leq C h^{-1}\eta \N{v_h}_{L^{2}(\Omega_h^n)}^2.
\end{equation}
The estimate for $\|v_h\|_{L^2(\Omega^n\backslash \Omega_{\eta}^n)}$ is  similar. 
The second inequality is a direct consequence of \cite[Eq. (17) in Lemma 10]{Nicaise}. 
\end{proof}

According to Lemma \ref{Lemma:4.4}, we have
\begin{equation*}
    \begin{aligned}
    \int_{\Omega_{\eta}^n\backslash \Omega^n}  \nabla \hat{u}^n \cdot \nabla e_{u,0}^n \mathrm{d} x  
    &= \int_{\Omega_{\eta}^n\backslash \Omega^n}  (\nabla \hat{u}^n -\nabla u^n) \cdot \nabla e_{u,0}^n \mathrm{d} x 
    +\int_{\Omega_{\eta}^n\backslash \Omega^n} \nabla u^n \cdot \nabla e_{u,0}^n \mathrm{d} x \\
    &\leq C h^{k-\frac{1}{2}}\eta^{\frac{1}{2}} \left| e_{u,0}^n \right|_{H^1(\Omega_{h}^n)} + C h^{-\frac{1}{2}}\eta \left|e_{u,0}^n\right|_{H^1(\Omega_{h}^n)} ,
    \end{aligned}
\end{equation*}
which yields 
\begin{equation}\nonumber
\begin{aligned}
   \text{I}_1^u &= \int_{\Omega_{\eta}^n\backslash \Omega^n}  \nabla \hat{u}^n \cdot \nabla e_{u,0}^n \mathrm{d} x - \int_{\Omega^n \backslash \Omega_{\eta}^n}  \nabla \hat{u}^n \cdot \nabla e_{u,0}^n \mathrm{d} x 
    \leq C \Big(h^{k-\frac{1}{2}}\eta^{\frac{1}{2}}+h^{-\frac{1}{2}}\eta \Big)\,\left\|e_{u,0}^n\right\|_{H^1(\Omega_h^n)}.
\end{aligned}
\end{equation}

To estimate $\mathrm{I}_3^u$, we use the deformation $\BT_\theta^n$ and the transport error $e_{u,\theta}^n$.
Since $u^n$ is independent of $\theta$ and $\partial^\bullet e_{u,\theta}^n=0$, combining \eqref{eq:DT_theta} and \eqref{grad_material} yields
\begin{equation}\label{eq:I3u_transport}
\begin{aligned}
\mathrm{I}_3^u
&=-\int_0^1 \frac{\D}{\D\theta}\int_{\Omega_\theta^n}\nabla u^n\cdot \nabla e_{u,\theta}^n\,\D x\,\D\theta \\
&=-\int_0^1\int_{\Omega_\theta^n}\Big( (\nabla \cdot \Be_{v,\theta}^n)\,\nabla u^n\cdot \nabla e_{u,\theta}^n
+(\nabla^2 u^n\,\Be_{v,\theta}^n)\cdot \nabla e_{u,\theta}^n
-(\nabla \Be_{v,\theta}^n)^{\top}\nabla e_{u,\theta}^n\cdot \nabla u^n\Big)\,\D x\,\D\theta \\
&\le C\int_0^1 \|\Be_{v,\theta}^n\|_{H^1(\Omega_\theta^n)}\,\|e_{u,\theta}^n\|_{H^1(\Omega_\theta^n)}\,\D\theta
\le C\|\Be_v^n\|_{H^1(\Omega^n)}\,\|e_{u,0}^n\|_{H^1(\Omega_h^n)},
\end{aligned}
\end{equation}
where in the last step, we used Lemma~\ref{lem:est T_theta} and (\ref{eq:ev_H1}) to obtain
\begin{equation}\label{eq:eu_theta_H1_bd}\nonumber
\begin{aligned}
\|e_{u,\theta}^n\|_{L^2(\Omega_\theta^n)}^2
&=\int_{\Omega^n}|e_{u,0}^n(\Bx)|^2\,|\det(\bbJ_{\theta}^n)|\,\D x
\le C\,\|e_{u,0}^n\|_{L^2(\Omega^n)}^2 \le C\,||e_{u,0}^n||_{L^2(\Omega_h^n)}^2, \\ 
|e_{u,\theta}^n|_{H^1(\Omega_\theta^n)}^2
&=\int_{\Omega^n}\big|\nabla e_{u,0}^n(\Bx)\,(\bbJ_\theta^n(\Bx))^{-1}\big|^2\,|\det(\bbJ_\theta^n)|\,\D x
\le C\,|e_{u,0}^n|_{H^1(\Omega^n)}^2
\le C\,|e_{u,0}^n|_{H^1(\Omega_h^n)}^2.
\end{aligned}
\end{equation}
Collecting \eqref{eq:Iu_split}--\eqref{eq:I3u_transport} and setting
\[
\mathrm{E}_h:=h^k+h^{k-\frac12}\eta^{\frac12}+h^{-\frac12}\eta,
\]
we arrive at
\begin{equation}\label{eq:Iu_final}
|\mathrm{I}^u|
\le C\Big(\mathrm{E}_h+\|\Be_v^n\|_{H^1(\Omega^n)}\Big)\,\|e_{u,0}^n\|_{H^1(\Omega_h^n)}.
\end{equation}

\paragraph{\textbf{Estimate of $\mathrm{II}^u$.}}
Proceeding analogously, we write
\[
\mathrm{II}^u
=\Big(\int_{\Omega_\eta^n}\hat u^n e_{u,0}^n\,\D x -\int_{\Omega^n} \hat u^n e_{u,0}^n\,\D x \Big) 
+\int_{\Omega^n}(\hat u^n-u^n)e_{u,0}^n\,\D x
+\Big(\int_{\Omega^n} u^n e_{u,0}^n\,\D x -\int_{\Omega(t_n)} u^n e_{u,1}^n\,\D x\Big),
\]
and obtain
\begin{equation}\label{eq:IIu_final}
|\mathrm{II}^u|
\le C\Big(h^{k+\frac12}\eta^{\frac12}+h^{-\frac12}\eta\Big)\|e_{u,0}^n\|_{H^1(\Omega_h^n)}
+Ch^k\|e_{u,0}^n\|_{L^2(\Omega_h^n)}
+C\|\Be_v^n\|_{H^1(\Omega^n)}\|e_{u,0}^n\|_{L^2(\Omega_h^n)}.
\end{equation}

\paragraph{\textbf{Estimate of $\mathrm{III}^u$.}}
Using the same transport argument as in \eqref{eq:I3u_transport}, we have
\begin{equation}\label{eq:IIIu_final}
\begin{aligned}
|\mathrm{III}^u|
&=\Big(\int_{\Omega(t_n)} f e_{u,1}^n\,\D x -\int_{\Omega^n} f e_{u,0 }^n\,\D x \Big)
+\Big(\int_{\Omega^n} f e_{u,0}^n\,\D x -\int_{\Omega_\eta^n} f e_{u,0}^n\,\D x \Big) \\
&\le \left|\int_0^1 \frac{\D}{\D\theta}\Big(\int_{\Omega_\theta^n} f\,e_{u,\theta}^n\,\D x\Big)\,\D\theta\right| + C h^{-\frac12}\eta \|e_{u,0}^n\|_{L^2(\Omega_h^n)} 
\le C \Big( \|\Be_v^n\|_{H^1(\Omega^n)} + h^{-\frac12}\eta \Big) \,\|e_{u,0}^n\|_{L^2(\Omega_h^n)}.
\end{aligned}
\end{equation}

\paragraph{\textbf{Estimate of $\mathrm{IV}^u$.}}
Before estimating the last term, we first recall the trace inequality. For any $v \in H^1(K)$, it holds that
\begin{equation*}
\|v\|_{L^2(E)}^2 \le C \Big( h^{-1} \|v\|_{L^2(K)}^2 + h |v|_{H^1(K)}^2 \Big).
\end{equation*}
By the definition of the jump, we have
\begin{equation*}
\|\llbracket \partial_{\bm n}^s v \rrbracket\|_{L^2(E)}^2
\le C \Big(
\|\partial_{\bm n}^s v\|_{L^2(E \cap K_1)}^2
+
\|\partial_{\bm n}^s v\|_{L^2(E \cap K_2)}^2
\Big).
\end{equation*}
Thus, it suffices to estimate the higher-order normal derivatives on the face $E$ in terms of the quantities defined on the neighboring elements. Applying the trace inequality with $v = \nabla^s e_{u,0}^n$ and $v = \nabla^s (\hat u^n - u^n)$ yields
\begin{equation}\label{trace-ineq}
\begin{aligned}
\|\partial_{\bm n}^s (\hat u^n - u^n)\|_{L^2(E)}^2 \le \|\nabla^s (\hat u^n - u^n)\|_{L^2(E)}^2
&\le C \Big(
h^{-1} \|\nabla^s (\hat u^n - u^n)\|_{L^2(K)}^2
+
h \, |\nabla^s (\hat u^n - u^n)|_{H^1(K)}^2
\Big), \\
\|\partial_{\bm n}^s e_{u,0}^n\|_{L^2(E)}^2 \le \|\nabla^s e_{u,0}^n\|_{L^2(E)}^2
&\le C \Big(
h^{-1} \|\nabla^s e_{u,0}^n\|_{L^2(K)}^2
+
h \, |\nabla^s e_{u,0}^n|_{H^1(K)}^2
\Big).
\end{aligned}
\end{equation}
Since $e_{u,0}^n$ is a finite element function, the following inverse inequalities hold:
\[
\|\nabla^s e_{u,0}^n\|_{L^2(K)}
\le C h^{1-s} |e_{u,0}^n|_{H^1(K)}, \qquad 
|\nabla^s e_{u,0}^n|_{H^1(K)}
\le C h^{-s} |e_{u,0}^n|_{H^1(K)}.
\]
Combining standard interpolation error estimates with \eqref{trace-ineq}, we obtain
\begin{equation*}
\begin{aligned}
\|\partial_{\bm n}^s (\hat u^n - u^n)\|_{L^2(E)}^2
&\le C \Big(
h^{-1} h^{2(k+1-s)}
+
h \cdot h^{2(k-s)}
\Big) |u^n|_{H^{k+1}(K)}^2 \le C h^{2k-2s+1} |u^n|_{H^{k+1}(K)}^2, \\
\|\partial_{\bm n}^s e_{u,0}^n\|_{L^2(E)}^2
&\le C \Big(
h^{-1} h^{2(1-s)} + h \cdot h^{-2s} \Big) |e_{u,0}^n|_{H^1(K)}^2 \le C h^{1-2s} |e_{u,0}^n|_{H^1(K)}^2.
\end{aligned}
\end{equation*}
Multiplying by the weight $h^{2s-1}$ and applying the Cauchy-Schwarz inequality, we deduce that
\begin{equation*}
h^{2s-1} \int_E \llbracket{\partial_{\bm{n}}^s (\hat u^n - u^n)}\rrbracket \llbracket{\partial_{\bm{n}}^s e_{u,0}^n}\rrbracket \, ds 
\le C (h^{4s-2} h^{2k-2s+1} h^{1-2s})^{1/2} |e_{u,0}^n|_{H^1(K)} 
\le C h^k |e_{u,0}^n|_{H^1(K)}.
\end{equation*}
Summing over all elements $K$, we conclude that
\begin{equation}\label{eq:IVu_final}
|\mathrm{IV}^u|=\big|\mathscr{J}_h^n(\hat u^n-u^n,e_{u,0}^n)\big|
\le Ch^k\,\|e_{u,0}^n\|_{H^1(\Omega_h^n)}.
\end{equation}

\paragraph{\textbf{Conclusion of \eqref{du}.}}
Combining \eqref{eq:Iu_final}--\eqref{eq:IVu_final} and using \eqref{vtophi_H1}, we finally obtain
\[
\int_{\Omega_\eta^n} d_u^n\, e_{u,0}^n\,\D x
\le C\Big(\mathrm{E}_h+\|\Be_\phi^n\|_{H^1(\Omega^0)}\Big)\,\|e_{u,0}^n\|_{H^1(\Omega_h^n)}.
\]

\subsubsection{Estimate for $d_p^n$}\label{sec:estimate for dp}
Subtracting \eqref{system2-p} from \eqref{eq:hat p}, after taking $\chi_p = e_p^n= e_{p,0}^n \in V(k,\Ct_h^n)$ in \eqref{eq:hat p} and $\chi_p = e_{p,1}^n \in H^1(\Omega(t_n))$ in \eqref{system2-p},
we obtain 
\begin{equation*}
\begin{aligned}
    &\int_{\Omega_{\eta}^n} d_p^n  e_{p,0}^n \D x = \int_{\Omega_{\eta}^n} \nabla \hat{p}^n \cdot \nabla e_{p,0}^n +\hat{p}^n e_{p,0}^n - (\hat{u}^n - u_d) e_{p,0}^n \mathrm{d} x + \mathscr{J}_h^n(\hat{p}^n,e_{p,0}^n) \\
  =&\underbrace{\bigg(\int_{\Omega_{\eta}^n} \nabla \hat{p}^n \cdot \nabla e_{p,0}^n \mathrm{d} x - \int_{\Omega(t_n)} \nabla p^n \cdot \nabla e_{p,1}^n \mathrm{d} x\bigg)}_{\text{I}^p}  +
  \underbrace{\bigg(\int_{\Omega_{\eta}^n} \hat{p}^n e_{p,0}^n \mathrm{d} x - \int_{\Omega(t_n)} p^n e_{p,1}^n \mathrm{d} x\bigg)}_{\text{II}^p}\\
  &+ \underbrace{\bigg( \int_{\Omega_{\eta}^n} u_d e_{p,0}^n \mathrm{d} x - \int_{\Omega(t_n)}  u_d e_{p,1}^n \mathrm{d} x \bigg)}_{\text{III}^p}+ 
  \underbrace{\bigg(\mathscr{J}_h^n(\hat{p}^n,e_{p,0}^n) - \mathscr{J}_h^n(p^n,e_{p,0}^n)\bigg)}_{\text{IV}^p}+
  \underbrace{\bigg(\int_{\Omega(t_n)} u^n e_{p,1}^n \D x-\int_{\Omega_{\eta}^n} \hat{u}^n e_{p,0}^n \D x \bigg)}_{\text{V}^p}.
\end{aligned}
\end{equation*}
The estimates for $\text{I}^p - \text{V}^p$ are similar to $\text{I}^u - \text{IV}^u$. 
Thus, we obtain the following consistency error of the adjoint state
\begin{equation}\nonumber
    \int_{\Omega_{\eta}^n} d_p^n  e_{p,0}^n \D x \leq C\Big(\text{E}_h+ \N{\Be_\phi^n}_{H^1(\Omega^0)}\Big) \N{e_{p,0}^n}_{H^1(\Omega_h^n)}.
\end{equation}

\subsubsection{Estimate for $\Bd_w^n$}\label{sec:estimate for dw}
Using \eqref{eq:hat w} and \eqref{system2-w}, and taking the test functions
$\chibf_w=\Be_{w,0}^n=\Be_w^n\in \bm V(k,\Ct_h^n)$ in \eqref{eq:hat w} and
$\chibf_w=\Be_{w,1}^n\in \BH^1(\Omega(t_n))$ in \eqref{system2-w}, we obtain
\begin{equation}\label{eq:dw_split}
\begin{aligned}
 &\int_{\Omega_{\eta}^n} \Bd_w^n \cdot \Be_{w,0}^n \mathrm{~d} x = \int_{\Omega_{\eta}^n} \nabla \hat{\Bw}^n: \nabla \Be_{w,0}^n+\hat{\Bw}^n \cdot \Be_{w,0}^n \mathrm{d} x + \mathscr{J}_h^n(\hat{\Bw}^n,\Be_{w,0}^n) + \mathrm{d} J\left( \Gamma_{\eta}^n, \hat{u}^n, \hat{p}^n ; \Be_{w,0}^n\right) \\
 =& \underbrace{\bigg(\int_{\Omega_{\eta}^n} \nabla \hat{\Bw}^n : \nabla \Be_{w,0}^n \mathrm{d} x - \int_{\Omega(t_n)} \nabla \Bw^n : \nabla \Be_{w,1}^n \mathrm{d} x\bigg)}_{\text{I}^w}  +
 \underbrace{\bigg(\int_{\Omega_{\eta}^n} \hat{\Bw}^n \cdot \Be_{w,0}^n \mathrm{d} x - \int_{\Omega(t_n)} \Bw^n \cdot \Be_{w,1}^n \mathrm{d} x\bigg)}_{\text{II}^w}\\
 &\;+\underbrace{ \bigg(\mathrm{d} J\left( \Gamma_{\eta}^n, \hat{u}^n, \hat{p}^n ; \Be_{w,0}^n\right) - \mathrm{d} J\left( \Gamma(t_n), u^n, p^n ; \Be_{w,1}^n\right)\bigg)}_{\text{III}^w} +\underbrace{\bigg( \mathscr{J}_h^n(\hat{\Bw}^n,\Be_{w,0}^n) - \mathscr{J}_h^n(\Bw^n,\Be_{w,0}^n)\bigg)}_{\text{IV}^w}.
\end{aligned}
\end{equation}
The terms $\mathrm{I}^w$, $\mathrm{II}^w$ and $\mathrm{IV}^w$ can be estimated in the same way as in
subsections~\ref{sec:estimate for du}--\ref{sec:estimate for dp}, which gives
\begin{equation}\label{eq:dw_linear_terms}\nonumber
|\mathrm{I}^w|+|\mathrm{II}^w|+|\mathrm{IV}^w|
\le C\Big(\mathrm{E}_h+\|\Be_v^n\|_{H^1(\Omega^n)}\Big)\,\|\Be_{w,0}^n\|_{\BH^1(\Omega_h^n)}.
\end{equation}
It remains to bound the shape derivative term $\mathrm{III}^w$.

\paragraph{\textbf{Estimate of $\mathrm{III}^w$.}}
Using \eqref{distributed_derivative}, we decompose $\mathrm{III}^w$ into terms of the form

\begin{equation*}\label{estimate_derivative}
\begin{aligned}
 \text{III}^w 
 = & \underbrace{\bigg(\int_{\Omega_{\eta}^n} 2 \nabla \hat{u}^n \cdot \bbD(\Be_{w,0}^n) \nabla \hat{p}^n  {\rm d} x - \int_{\Omega(t_n)} 2\nabla u^n \cdot \bbD(\Be_{w,1}^n) \nabla p^n {\rm d} x) \bigg)}_{\text{III}_1^w}    \\                                
 &\ + \int_{\Omega_{\eta}^n}  \hat{p}^n \nabla f \cdot \Be_{w,0}^n {\rm d}x-\int_{\Omega(t_n)}p^n \nabla f \cdot \Be_{w,1}^n {\rm d}x\\
 &\ + \int_{\Omega_{\eta}^n} \Big(\frac{1}{2}|\hat{u}^n - u_{d}|^{2} - \nabla \hat{u}^n \cdot \nabla \hat{p}^n - \hat{u}^n \hat{p}^n + f\hat{p}^n \Big) \nabla \cdot \Be_{w,0}^n {\rm d} x - \int_{\Omega_{\eta}^n} (\hat{u}^n - u_{d})\nabla u_d \cdot \Be_{w,0}^n {\rm d} x\\
   &\ -\int_{\Omega(t_n)} \Big(\frac{1}{2}|u^n - u_{d}|^{2} - \nabla u^n \cdot \nabla p^n - u^n p^n + fp^n \Big) \nabla \cdot \Be_{w,1}^n {\rm d} x  + \int_{\Omega(t_n)} (u^n - u_{d})\nabla u_d \cdot \Be_{w,1}^n {\rm d} x.
\end{aligned}
\end{equation*}
In fact, taking the first term as an example and using the interpolation error estimate, we have 
\begin{equation}\nonumber
    \begin{aligned}
        \text{III}_1^w 
        = &\underbrace{\int_{\Omega_{\eta}^n} 2\nabla (\hat{u}^n-u^n)  \cdot \bbD(\Be_{w,0}^n) \nabla \hat{p}^n \D x}_{\text{III}_{1,1}^w}
         + \underbrace{\int_{\Omega_{\eta}^n} 2\nabla u^n \cdot \bbD(\Be_{w,0}^n) \nabla \hat{p}^n \D x
        -\int_{\Omega^n} 2\nabla u^n \cdot \bbD(\Be_{w,0}^n) \nabla \hat{p}^n \D x}_{\text{III}_{1,2}^w}\\
        &\ +\underbrace{\int_{\Omega^n} 2\nabla u^n \cdot \bbD(\Be_{w,0}^n)\nabla (\hat{p}^n-p^n) \D x}_{\text{III}_{1,3}^w}
         +\underbrace{\int_{\Omega^n} 2 \nabla u^n \cdot \bbD(\Be_{w,0}^n) \nabla p^n \D x - \int_{\Omega(t_n)} 2 \nabla u^n \cdot\bbD(\Be_{w,1}^n)\nabla p^n \D x}_{\text{III}_{1,4}^w} \\
        \leq\,&  C h^k \SN{\Be_{w,0}^n}_{\BH^1(\Omega_h^n)} + 
        \text{III}_{1,2}^w+\text{III}_{1,4}^w.
    \end{aligned}
\end{equation}
The second term $\text{III}_{1,2}^w$ in this formula can be estimated by H$\ddot{\mathrm{o}}$lder's inequality and Lemma \ref{Lemma:4.4} as
\begin{equation}\label{eq:III1w_interp}
\begin{aligned}
\int_{\Omega_{\eta}^n \backslash \Omega^n} 2\nabla u^n \cdot \bbD(\Be_{w,0}^n) \nabla \hat{p}^n{\rm d} x
    =& \int_{\Omega_{\eta}^n \backslash \Omega^n} 2\nabla u^n \cdot \bbD(\Be_{w,0}^n) (\nabla \hat{p}^n -\nabla p^n){\rm d} x + \int_{\Omega_{\eta}^n \backslash \Omega^n} 2 \nabla u^n \cdot \bbD(\Be_{w,0}^n)\nabla p^n{\rm d} x\\
    \leq& C ||\nabla u^n||_{L^\infty(D)} ||\bbD(\Be_{w,0}^n)||_{\mathbb{L}^2(\Omega_{\eta}^n \backslash \Omega^n)} \left(||\nabla \hat{p}^n - \nabla p^n||_{L^2(\Omega_{\eta}^n \backslash \Omega^n)} + ||\nabla p^n||_{L^2(\Omega_{\eta}^n \backslash \Omega^n)}\right) \\
    \leq & C \Big(h^{k-\frac 12}\eta^{\frac 12} + h^{-\frac 12}\eta\Big)\SN{\Be_{w,0}^n}_{\BH^1(\Omega_h^n)}.
\end{aligned}
\end{equation}
The other part follows by a similar argument. For the last bracket in $\text{III}_{1,4}^w$ , we employ the domain deformation
$\BT_\theta^n$ and define $\Be_{w,\theta}^n$ by transport, i.e.,
$\Be_{w,\theta}^n\circ \BT_\theta^n=\Be_{w,0}^n$. In particular, $\partial^\bullet \Be_{w,\theta}^n=0$.
Then, using the transport formula and the identities
$\partial^\bullet(\nabla u^n)=\nabla^2 u^n\,\Be_{v,\theta}^n$,
$\partial^\bullet(\nabla p^n)=\nabla^2 p^n\,\Be_{v,\theta}^n$, and
$\partial^\bullet(\nabla \Be_{w,\theta}^n)=-(\nabla \Be_{v,\theta}^n)^{\top}\nabla\Be_{w,\theta}^n$,
we obtain
\begin{equation}\label{eq:III1w_transport}
\begin{aligned}
   \text{III}_{1,4}^w 
    = &-\int_0^1 \frac{\D}{\D \theta} \int_{\Omega_{\theta}^n} 2 \nabla u^n \cdot \bbD(\Be_{w,\theta}^n) \nabla p^n \,\D x \D\theta \\
    =& -2\int_0^1 \int_{\Omega_{\theta}^n} (\nabla \cdot \Be_{v, \theta}^n) \nabla u^n \cdot \bbD(\Be_{w,\theta}^n) \nabla p^n  + \partial^{\bullet} (\nabla u^n) \cdot \bbD(\Be_{w,\theta}^n) \nabla p^n \\
    & \quad \quad \quad \quad + \nabla u^n \cdot \partial^{\bullet} \bbD(\Be_{w,\theta}^n)\nabla p^n + \nabla u^n \cdot \bbD(\Be_{w,\theta}^n) \partial^{\bullet} (\nabla p^n) \D x \D \theta\\
    \leq & C\|\Be_v^n\|_{H^1(\Omega^n)}\,\|\Be_{w,0}^n\|_{\BH^1(\Omega_h^n)}. 
\end{aligned}
\end{equation}

Combining \eqref{eq:III1w_interp}--\eqref{eq:III1w_transport} yields
\[
|\mathrm{III}_1^w|
\le C\Big(\mathrm{E}_h+\|\Be_v^n\|_{H^1(\Omega^n)}\Big)\,\|\Be_{w,0}^n\|_{\BH^1(\Omega_h^n)}.
\]
All the remaining contributions in $\mathrm{III}^w$ can be treated analogously, and we conclude that
\begin{equation}\label{eq:IIIw_final}
|\mathrm{III}^w|
\le C\Big(\mathrm{E}_h+\|\Be_v^n\|_{H^1(\Omega^n)}\Big)\,\|\Be_{w,0}^n\|_{\BH^1(\Omega_h^n)}.
\end{equation}

\paragraph{\textbf{Conclusion of \eqref{eq:dw_split}.}} Thus, we can obtain the consistency error for the velocity variable from \eqref{vtophi_H1}:
\begin{equation*}
    \int_{\Omega_{\eta}^n} \Bd_w^n  \Be_{w,0}^n \D x \leq C\left(\text{E}_h+ \N{\Be_\phi^n}_{H^1(\Omega^0)}\right) \N{\Be_{w,0}^n}_{\BH^1(\Omega_h^n)}.
\end{equation*}
\begin{Lemma}\label{lem:consis est}
    Under the assumption of Theorem \ref{Thm:4.1} and the property \eqref{assum}, we denote  $\text{E}_h = h^k + h^{k-\frac{1}{2}}\eta^{\frac 12} + h^{-\frac{1}{2}}\eta$. 
    The consistency estimates satisfy
    \begin{equation}\nonumber
    \begin{aligned}
        \int_{\Omega_{\eta}^n} d_u^n e_u^n \D x &\leq C\left(\text{E}_h + \|\Be_\phi^n\|_{H^1(\Omega^0)}\right) \N{e_u^n}_{H^1(\Omega_h^n)},\\
      \int_{\Omega_{\eta}^n} d_p^n  e_{p}^n \D x &\leq C\left(\text{E}_h+ \N{\Be_\phi^n}_{H^1(\Omega^0)}\right) \N{e_{p}^n}_{H^1(\Omega_h^n)},\\
      \int_{\Omega_{\eta}^n} \Bd_w^n  \Be_w^n \D x &\leq C\left(\text{E}_h+ \N{\Be_\phi^n}_{H^1(\Omega^0)}\right) \N{\Be_{w}^n}_{\BH^1(\Omega_h^n)},
    \end{aligned}
    \end{equation}
    where $\Be_\phi^n$, $e_u^n$, $e_p^n$ and $\Be_w^n$ are defined in \eqref{eq:sta err}.
    Here  $C>0$ is independent of $\tau$, $h$ and $n$.
\end{Lemma}

\subsection{Stability estimate}\label{sec:Stability estimate}
In this subsection, our aim is to derive bounds for $e_u^n$, $e_p^n$ and $\Be_w^n$ based on the consistency estimates.
Subtracting (\ref{interpolation_weak}) from (\ref{fe_weak}), we arrive at the following equalities:
\begin{equation*}
\begin{aligned}
  0 &= \int_{\Omega_{\eta}^n} \Big( \nabla e_u^n \cdot \nabla e_u^n + e_u^n e_u^n + d_u^n  e_u^n \Big) \D x + \mathscr{J}_h(e_u^n, e_u^n),\\
  0 &= \int_{\Omega_{\eta}^n} \left( \nabla e_p^n \cdot \nabla e_p^n +  e_p^n e_p^n + d_p^n  e_p^n \right) \D x + \mathscr{J}_h(e_p^n, e_p^n) - \int_{\Omega_{\eta}^n} j_u^{\prime}\left(x, u_h^n\right) e_p^n \mathrm{d} x + \int_{\Omega_{\eta}^n} j_u^{\prime}\left(x, \hat{u}^n\right) e_p^n \mathrm{d} x,\\
  0 &= \int_{\Omega_{\eta}^n} \Big( \nabla \Be_w^n : \nabla \Be_w^n + \Be_w^n \cdot \Be_w^n + \Bd_w^n \cdot \Be_w^n \Big) \D x + \mathscr{J}_h(\Be_w^n, \Be_w^n) + \mathrm{d} J\left( \Gamma_{\eta}^n, u_h^n, p_h^n ; \Be_w^n\right) - \mathrm{d} J\left( \Gamma_{\eta}^n, \hat{u}^n, \hat{p}^n ; \Be_w^n\right).
\end{aligned}
\end{equation*}
Then using \eqref{eq:norm inequality-h1} and \eqref{eq:norm inequality-l2}, we can obtain the following stability results
\begin{align*}
    &C||e_u^n||^2_{H^1(\Omega_h^n)} \leq  \int_{\Omega_{\eta}^n} \Big( \nabla e_u^n \cdot \nabla e_u^n + e_u^ne_u^n \Big) \D x + \mathscr{J}_h(e_u^n, e_u^n)  \leq  \Big| \int_{\Omega_{\eta}^n} d_u^n  e_u^n \D x \Big|, \\
    &C||e_p^n||^2_{H^1(\Omega_h^n)} \leq   \int_{\Omega_{\eta}^n} \left( j_u^{\prime}\left(x, u_h^n\right) e_p^n - j_u^{\prime}\left(x, \hat{u}^n\right) e_p^n \right) \mathrm{d} x - \int_{\Omega_{\eta}^n} d_p^n  e_p^n \D x  \leq   \SN{\int_{\Omega_{\eta}^n} d_p^n  e_p^n \D x } + \N{e_u^n}_{L^2(\Omega_{\eta}^n)}\N{e_p^n}_{L^2(\Omega_{\eta}^n)} ,\\
    &C||\Be_w^n||^2_{\BH^1(\Omega_h^n)} \leq  \mathrm{d} J\left( \Gamma_{\eta}^n, \hat{u}^n, \hat{p}^n ; \Be_w^n\right) - \mathrm{d} J\left( \Gamma_{\eta}^n, u_h^n, p_h^n ; \Be_w^n\right) - \int_{\Omega_{\eta}^n} \Bd_w^n  \Be_w^n \D x.
\end{align*}
Observing the formulation of the Eulerian derivative in \eqref{distributed_derivative}, we have
\begin{equation} \nonumber
\begin{aligned}
    &\mathrm{d} J\left( \Gamma_{\eta}^n, u_h^n, p_h^n ;\Be_w^n\right) - \mathrm{d} J\left( \Gamma_{\eta}^n, \hat{u}^n, \hat{p}^n ;\Be_w^n\right) \\    
    =  & \int_{\Omega_{\eta}^n} 2\nabla u_h^n \cdot \bbD(\Be_w^n) \nabla (p_h^n - \hat{p}^n) {\rm d} x + \int_{\Omega_{\eta}^n} 2\nabla (u_h^n - \hat{u}^n) \cdot \bbD(\Be_w^n) \nabla \hat{p}^n {\rm d} x + \int_{\Omega_{\eta}^n} e_p^n \nabla f \cdot \Be_w^n {\rm d} x   \\
    &\ +     \int_{\Omega_{\eta}^n} \Big(\frac{1}{2}(u_h^n + \hat{u}^n)(u_h^n - \hat{u}^n) - (u_h^n - \hat{u}^n) u_d - \nabla u_h^n \cdot \nabla (p_h^n - \hat{p}^n) - \nabla (u_h^n - \hat{u}^n) \cdot \nabla \hat{p}^n \Big) \nabla \cdot \Be_w^n {\rm d} x \\
    &\ +     \int_{\Omega_{\eta}^n} \Big(-u_h^n (p_h^n - \hat{p}^n) -  (u_h^n - \hat{u}^n) \hat{p}^n + fe_p^n \Big) \nabla \cdot \Be_w^n {\rm d} x - \int_{\Omega_{\eta}^n} e_u^n \nabla u_d \cdot \Be_w^n {\rm d} x                                                             \\
    =   &   \int_{\Omega_{\eta}^n} 2\nabla u_h^n \cdot \bbD(\Be_w^n) \nabla e_p^n {\rm d} x + \int_{\Omega_{\eta}^n} 2\nabla e_u^n \cdot \bbD(\Be_w^n) \nabla \hat{p}^n {\rm d} x + \int_{\Omega_{\eta}^n} e_p^n \nabla f \cdot \Be_w^n {\rm d} x  - \int_{\Omega_{\eta}^n} e_u^n \nabla u_d \cdot \Be_w^n {\rm d} x \\
    &\ +   \int_{\Omega_{\eta}^n} \Big(\frac{1}{2}(u_h^n + \hat{u}^n)e_u^n - e_u^n u_d - \nabla u_h^n \cdot \nabla e_p^n - \nabla e_u^n \cdot \nabla \hat{p}^n - u_h^n e_p^n - e_u^n \hat{p}^n + fe_p^n \Big) \nabla \cdot \Be_w^n {\rm d} x .  
\end{aligned}
\end{equation}
Using condition (\ref{eq:bd of uh}), the velocity error can be transformed as
\begin{equation}\nonumber
\begin{aligned}
    \N{\Be_w^n}^2_{\BH^1(\Omega_h^n)} &\leq C\left( \mathrm{d} J\left( \Gamma_{\eta}^n, \hat{u}^n, \hat{p}^n ; \Be_w^n\right) - \mathrm{d} J\left( \Gamma_{\eta}^n, u_h^n, p_h^n ; \Be_w^n\right) - \int_{\Omega_{\eta}^n} \Bd_w^n  \Be_w^n \D x\right)\\
    &\leq C \left(\left(\N{e_u^n}_{H^1(\Omega_{\eta}^n)} + \N{e_p^n}_{H^1(\Omega_{\eta}^n)} \right) \N{\Be_w^n}_{\BH^1(\Omega_{\eta}^n)} + \Big| \int_{\Omega_{\eta}^n} \Bd_w^n  \Be_w^n \D x \Big|\right).
\end{aligned}
\end{equation}

Based on the above discussion, we obtain the following stability estimates.
\begin{Lemma}\label{lem:sta est}
    Under the assumptions in Theorem \ref{Thm:4.1}, for any $0 \leq n \leq N$, the stability estimates of $u$ and $p$ satisfy
    \begin{subequations}\nonumber
    \begin{align}
      \N{e_u^n}_{H^1(\Omega_h^n)}^2 & \leq C \SN{\int_{\Omega_{\eta}^n} d_u^n e_u^n \D x},\\
      \N{e_p^n}_{H^1(\Omega_h^n)}^2 & \leq C  \left( \N{e_u^n}_{L^2(\Omega_{\eta
      }^n)}\N{e_p^n}_{L^2(\Omega_{\eta}^n)} + \SN{\int_{\Omega_{\eta}^n} d_p^n e_p^n \D x} \right).\\
      \end{align}
    \end{subequations}
    Additionally, with \eqref{eq:bd of uh}, for $0 \leq n \leq m-1$ the stability estimate of $\Bw$ satisfies
    \begin{equation}\nonumber
      \N{\Be_w^n}^2_{\BH^1(\Omega_h^n)} \leq  C \left(\left(\N{e_u^n}_{H^1(\Omega_{\eta}^n)} + \N{e_p^n}_{H^1(\Omega_{\eta}^n)} \right) \N{\Be_w^n}_{\BH^1(\Omega_{\eta}^n)} + \Big| \int_{\Omega_{\eta}^n} \Bd_w^n  \Be_w^n \D x \Big|\right),
    \end{equation}
    where $C>0$ is independent of $\tau$, $h$ and $n$.
\end{Lemma}

\subsection{Proof of Theorem \ref{Thm:4.1}}
Recall that the spatial error is defined as $\text{E}_h = h^k + h^{k-\frac 12}\eta^{\frac 12} + h^{-\frac 12}\eta$. By combining the results of Lemmas~\ref{lem:consis est} and~\ref{lem:sta est}, we derive the following estimates:
\begin{subequations} \nonumber
\begin{align}
    \N{e_{u}^n}^2_{H^1(\Omega_h^n)} &
    \leq C\left(\text{E}_h+ \N{\Be_\phi^n}_{H^1(\Omega^0)}\right) \N{e_{u}^n}_{H^1(\Omega_h^n)}, \\
    \N{e_{p}^n}^2_{H^1(\Omega_h^n)} &
     \leq C\left(\text{E}_h  + \N{\Be_\phi^n}_{H^1(\Omega^0)}\right) \N{e_{p}^n}_{H^1(\Omega_h^n)},\\
    \N{\Be_w^n}^2_{\BH^1(\Omega_h^n)} &
    \leq C\left(\text{E}_h + \N{\Be_\phi^n}_{H^1(\Omega^0)}\right) \N{\Be_w^n}_{\BH^1(\Omega_h^n)}.
\end{align}
\end{subequations}
The last inequality immediately implies 
\begin{equation}\nonumber
    \N{\Be_w^n}_{H^1(\Omega_h^n)} \leq C\left(\text{E}_h + \N{\Be_\phi^n}_{H^1(\Omega^0)} \right)
\end{equation}
for $0 \leq n \leq m-1$. 
Recalling the estimate for $\Be_\phi^n$ from Lemma~\ref{lem: ephi}, we obtain
\begin{equation}\nonumber
\begin{aligned}
    \N{\Be_\phi^n}_{H^1(\Omega^0)} &\leq  C\tau+ C\text{E}_h+ C\tau \sum_{l=1}^n \N{\Be_{\phi}^{l-1}}_{H^1(\Omega^0)}.
\end{aligned}
\end{equation}
Applying a discrete Gronwall's inequality yields
\begin{equation}\nonumber
\begin{aligned}
    \N{\Be_\phi^n}_{H^1(\Omega^0)} &\leq  C\left(\tau + \text{E}_h \right)
\end{aligned}
\end{equation}
 for $n=m$. Substituting this into the previous estimates gives the following
\begin{equation*}
    \begin{aligned}
      \N{e_u^n}_{H^1(\Omega_h^n)} \leq & C (\tau+\text{E}_h), \quad 
      \N{e_p^n}_{H^1(\Omega_h^n)} \leq  C (\tau+\text{E}_h),\quad 
      \N{\Be_w^n}_{\BH^1(\Omega_h^n)} \leq & C (\tau+\text{E}_h)
    \end{aligned}
\end{equation*}
for $0 \leq n \leq m$.
Since $k \geq 2$, there exists $h_0>0$, we may choose $\tau = o(h)$ and $\eta = o(h^{\frac{3}{2}})$ with $0<h\leq h_0$,  such that 
\begin{equation}\nonumber
\begin{aligned}
    &\N{\Be_\phi^n}_{W^{1,\infty}(\Omega^0)} \leq Ch^{-1} \N{\Be_\phi^n}_{H^1(\Omega^0)} \leq Ch^{-1}(\tau + \text{E}_h ) \leq \min \left\{1, \frac{1}{2C_\phi}\right\}, \\
    &\N{e_{u} ^n}_{W^{1,\infty}(\Omega_h^n)}  \leq Ch^{-1} \N{e_{u}^n}_{H^1(\Omega_h^n)} \leq Ch^{-1}(\tau + \text{E}_h ) \leq 1, \\
     &\N{e_{p} ^n}_{W^{1,\infty}(\Omega_h^n)}  \leq Ch^{-1} \N{e_{p}^n}_{H^1(\Omega_h^n)} \leq Ch^{-1}(\tau + \text{E}_h ) \leq 1,\\
    &\N{\Be_w^n}_{\BW^{1,\infty}(\Omega_h^n)}  \leq Ch^{-1} \N{\Be_w^n}_{\BH^1(\Omega_h^n)} \leq Ch^{-1}(\tau + \text{E}_h) \leq 1,
\end{aligned}
\end{equation}
which completes the proof.

\section{Numerical Experiments}\setcounter{equation}{0}
In this section, we report some numerical experiments to validate the theoretical results established in this work. The unfitted finite element method is implemented in MATLAB on a regular mesh. 
In all computations, the parameter in 
the ghost-penalty stabilization term \eqref{eq:ghost penalty} is set to $\alpha =1$.
The corresponding shape optimization algorithm is described below.
\begin{Algorithm}\label{Alg:5.1}\textbf{Shape steepest descent algorithm}
\begin{enumerate}
\item \textbf{Require}: {$f$, $u_d$, $\phi_{0}$}   
  \item  {$t \gets 0$}
  \item  \textbf{repeat}
  \item  \quad
solve the state and adjoint equations (\ref{discrete_u}) and (\ref{discrete_p}). 
\item  \quad
compute the descent direction by solving the Hilbertian regularization equation (\ref{discrete_w}).
\item \quad
 update the control points via \eqref{control_point} to obtain the corresponding cubic spline.
\item \quad $t \gets t + \tau$
\item  
\textbf{until} {$t = T$} 
\end{enumerate}
\end{Algorithm}

We demonstrate the convergence of the proposed method and its robustness in simulating the evolution of deformed boundaries. In numerical experiments, we set the computational domain $D =(0,3)^2$ centered at $\Bc = (1.5, 1.5)$. 

\begin{Example} \label{example1}
    The source function and shape density are set as
    \begin{align*}
        f(\Bx)&= ((x_1-c_1)^2+(x_2-c_2)^2)^2-18((x_1-c_1)^2+(x_2-c_2)^2)+9,\\
        j(\cdot,u)&= \frac{1}{2}|u-u_d|^2\quad \text{with} \;u_d = (1-(x_1-c_1)^2-(x_2-c_2)^2)^2.
    \end{align*}
    Then the optimal domain is the unit circle $\{(x_1, x_2):\ (x_1-c_1)^2+(x_2-c_2)^2 < 1 \}$. The initial domain is chosen as an ellipse with semi-major axis 1.3 and semi-minor axis 0.85. 
\end{Example}

According to the assumptions in Theorem \ref{Thm:4.1}, we choose $\tau = \mathcal{O}(h^{k})$ and $\eta = \mathcal{O}(h^{k+\frac{1}{2}})$. In numerical experiments, the final time is set to $T = 40$. Figure \ref{ex1} illustrates the evolution of the domain at $t=0$ and $t=40$.  To evaluate the convergence rates, we employ a sequence of mesh sizes and time steps defined by
\begin{equation}\nonumber
h_i = \frac{3}{2^{3+i}}, \quad \tau_i = \frac{T}{625 \times 4^{i}}, \quad i = 0, 1, 2, 3, 4.
\end{equation}
The term $\Be_\phi^n$ is evaluated using the geometric error defined as follows:
\begin{equation} \nonumber
    \sum_{K \in \Ct_h} \Big|\mbox{area}(\Omega(T) \cap K) - \mbox{area}(\Omega_\eta^N \cap K) \Big|.
\end{equation}
Here, $\Omega(T)$ denotes the optimal domain and $\Omega_\eta^N$ represents the cubic spline reconstruction of the discrete solution. Using the numerical solution at $h = \frac{3}{128}$ as a reference, the convergence results for $k = 2$ are presented in Table \ref{t1}. 

\begin{figure}[!htbp]
  \centering
  \subfigure[initial domain at \ $t = 0$]{
  \includegraphics[width=2.5in]{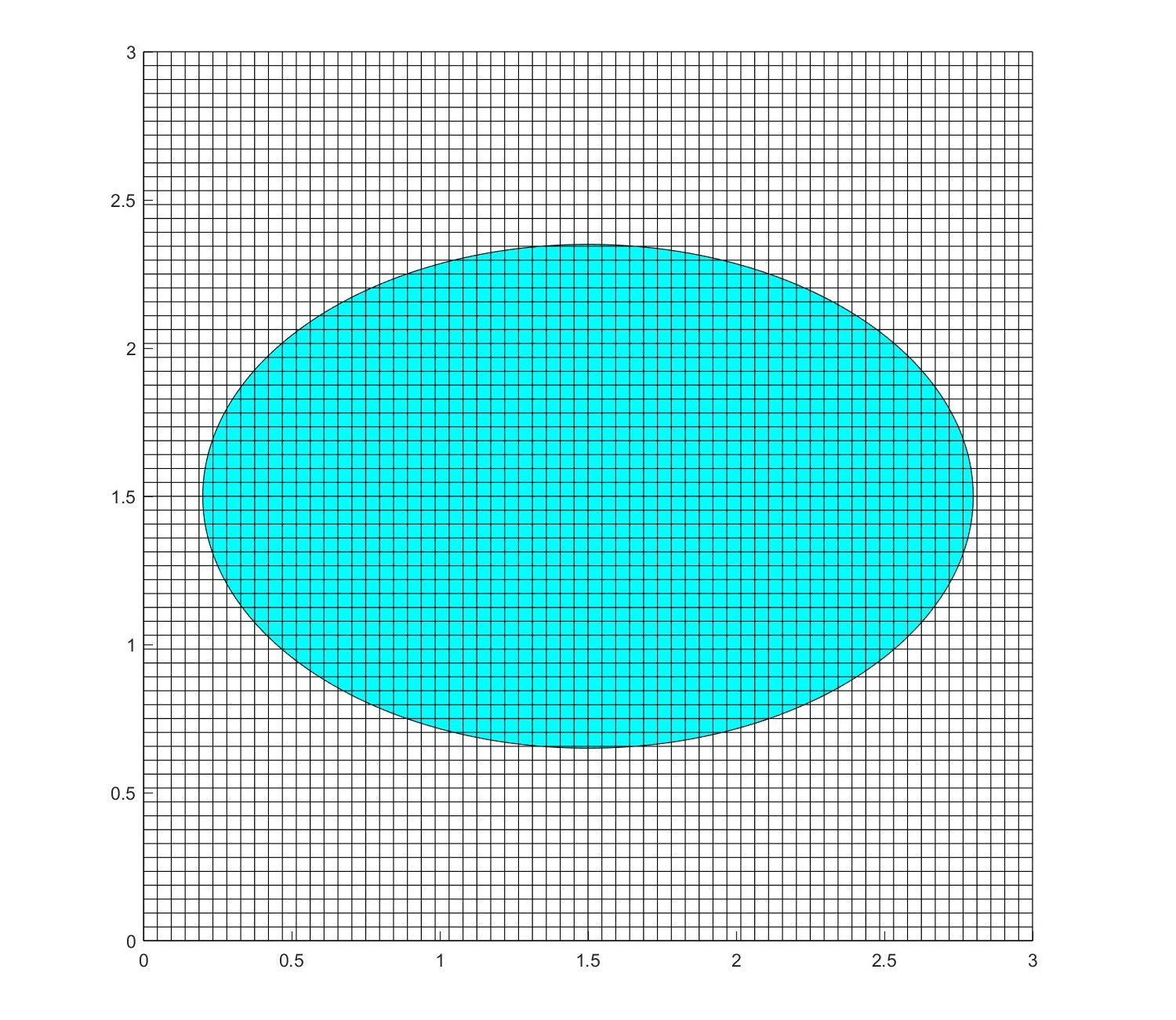}}\hspace{10mm}
  \subfigure[optimal domain at \ $t = 40$]{
  \includegraphics[width=2.5in]{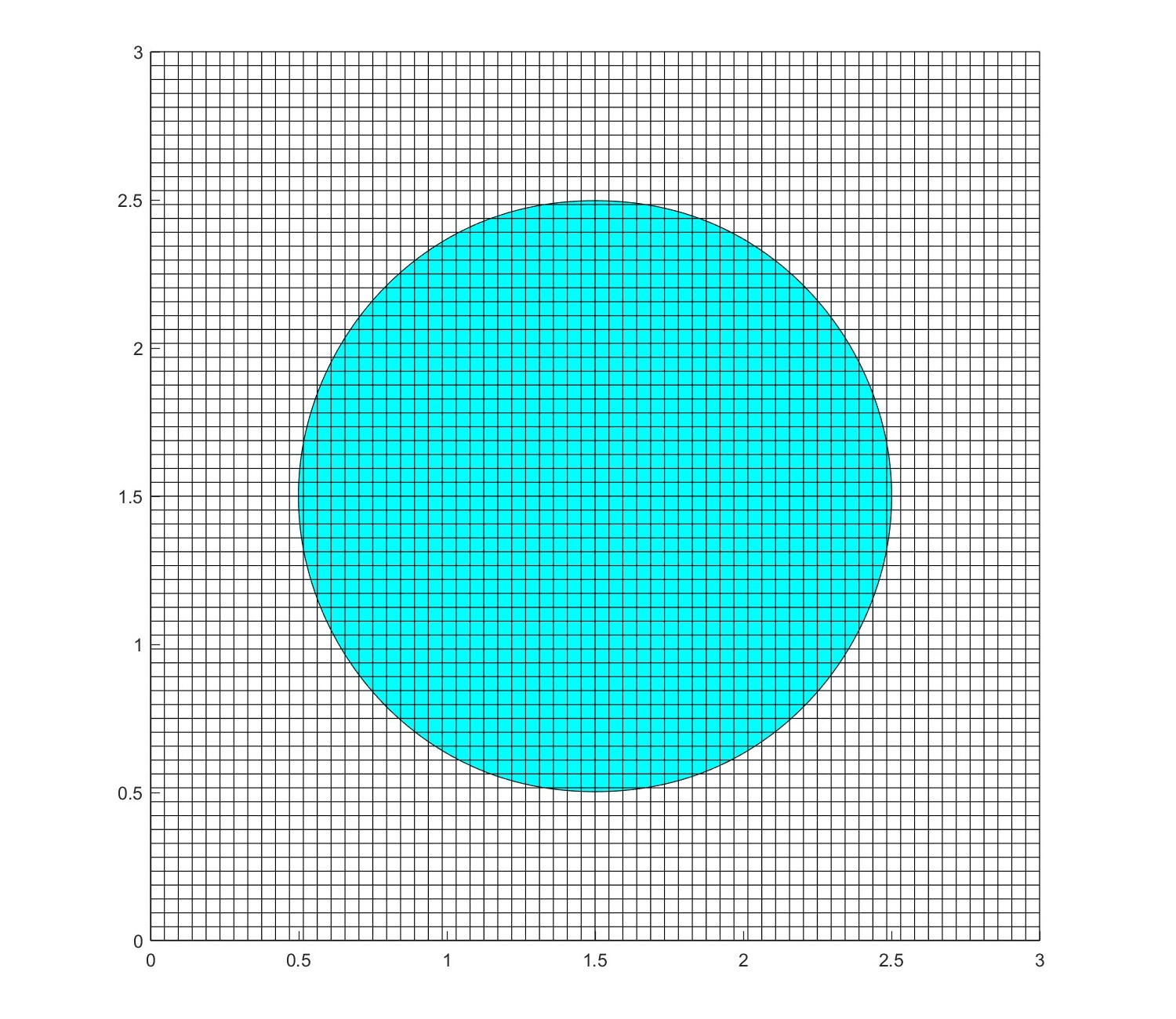}}
  \caption{Evolution of the domain for Example \ref{example1}.}
  \label{ex1}
\end{figure}

\begin{table}[h]
\centering
\renewcommand{\arraystretch}{1.3}
\setlength{\tabcolsep}{15pt}
\begin{tabular}{c c c c c}
\toprule
{$h$} & {order of $\phi$} & {order of $u$} & {order of $p$} & {order of $\Bw$} \\
\midrule
1/8  & - & - & - & - \\
1/16 & 3.9905 & 2.0207 & 3.7859  & 3.8351  \\
1/32 & 0.9245 & 2.3419 & 3.0753 & 2.3673 \\
1/64 & 1.7174 & 2.2618  & 1.5532 & 1.6211 \\
\bottomrule
\end{tabular}
\vspace{5pt}
\caption{Convergence orders of different variables for Example \ref{example1}.}
\label{t1}
\end{table}

\begin{Example}\label{example2}
    In this example the boundary of the initial domain is described by the following parametric equations for \(\theta \in [0, 2\pi]\)
    \[
    x_1(\theta) = c_1 + 1.2 \left( 2 \Big/ \bigg(8 + 6\sin\big(3\theta + \dfrac{3\pi}{36}\big)\bigg) \right)^{\!\!1/6} \cos \theta; \quad 
    x_2(\theta) = c_2 + 1.2 \left( 2 \Big/ \bigg(8 + 6\sin\big(3\theta + \dfrac{3\pi}{36}\big)\bigg) \right)^{\!\!1/6} \sin \theta.
    \]
All other settings remain the same as in Example \ref{example1}.
\end{Example}

\begin{figure}[!htbp]
  \centering
  \subfigure[initial domain at \ $T = 0$]{
  \includegraphics[width=2.5in]{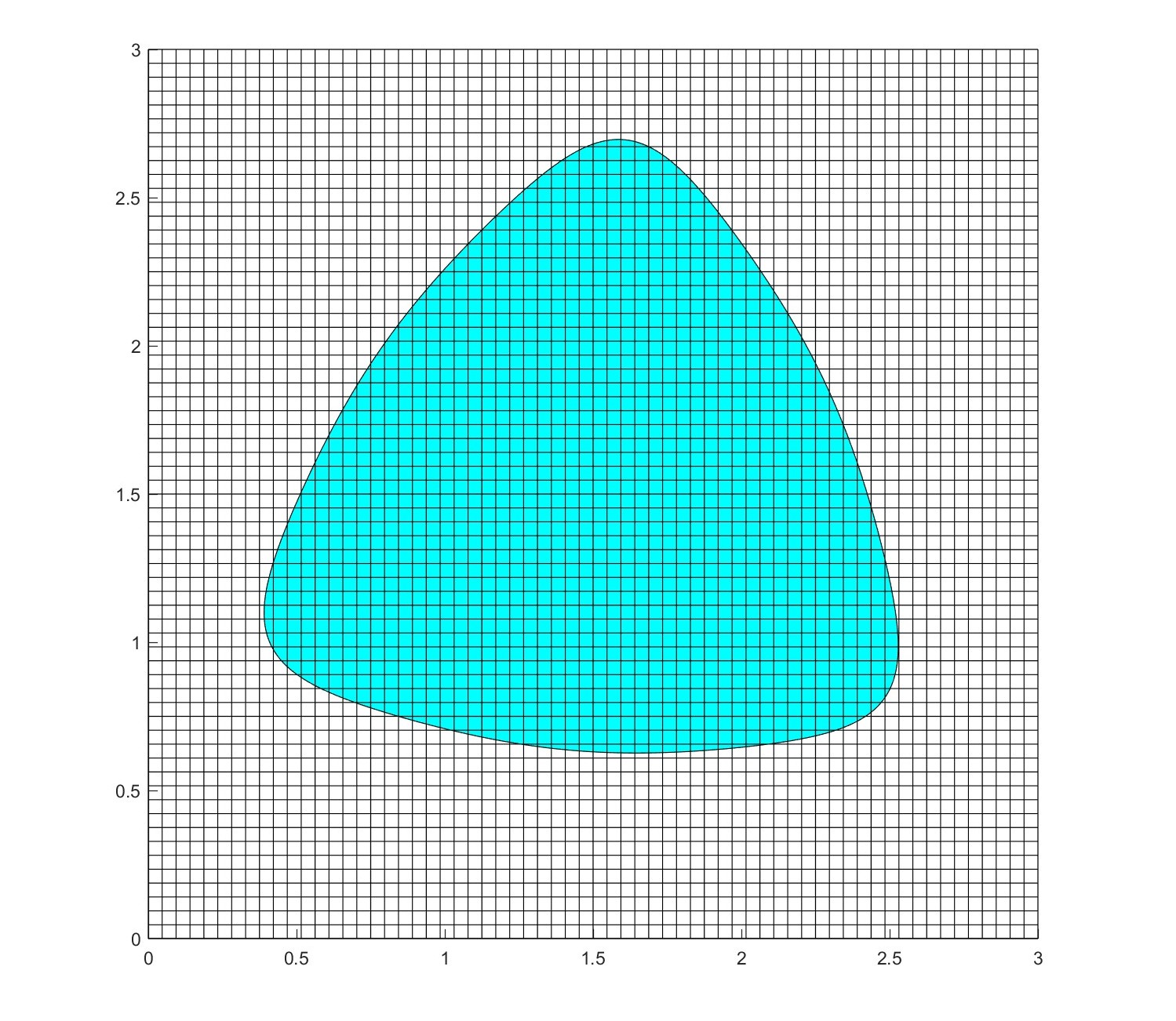}}\hspace{10mm}
  \subfigure[optimal domain at \ $T = 40$]{
  \includegraphics[width=2.5in]{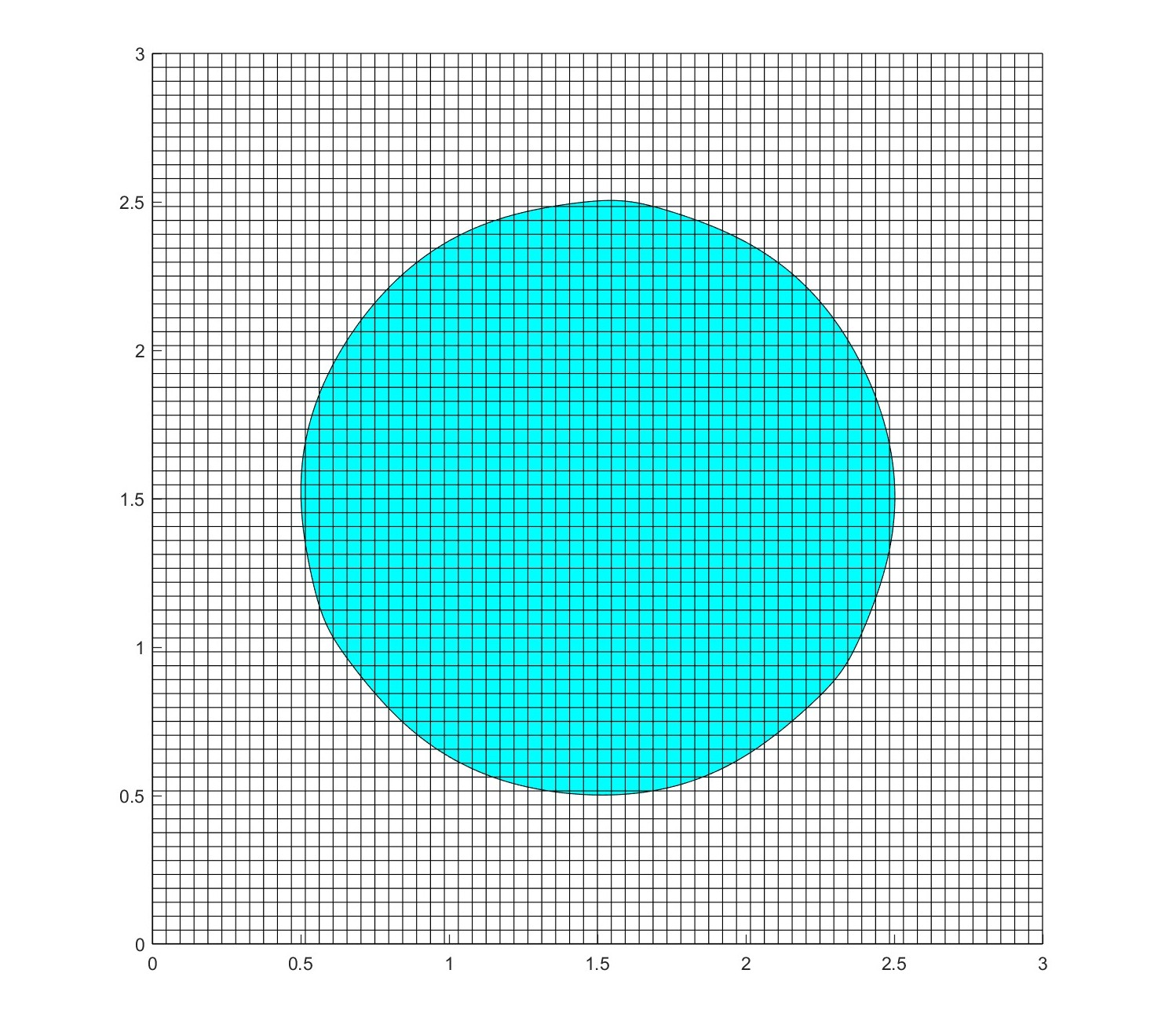}}
  \caption{Evolution of the domain for Example \ref{example2}.}
  \label{ex2}
\end{figure}

\begin{table}[h]
\centering
\renewcommand{\arraystretch}{1.3}
\setlength{\tabcolsep}{15pt}
\begin{tabular}{c c c c c}
\toprule
{$h$} & {order of $\phi$} & {order of $u$} & {order of $p$} & {order of $\Bw$} \\
\midrule
1/8  & - & - & - & -  \\
1/16 & 2.0425 & 2.0045 & 3.6323 & 3.3478 \\
1/32 & 1.7642 & 2.2168 & 2.1056 & 1.9336 \\
1/64 & 1.8626 & 2.0312 & 1.9120 & 1.9990 \\
\bottomrule
\end{tabular}
\vspace{5pt}
\caption{Convergence orders of different variables for Example \ref{example2}.}
\label{t2}
\end{table}

Figure \ref{ex2} shows the evolution of the domain, while the specific convergence orders are detailed in Table \ref{t2}. We observe that although the convergence order may deviate from the theoretical result in coarser meshes, it approaches the expected rate of 2 as the mesh size $h$ decreases. Furthermore, the results in Example \ref{example2} surpass those of Example \ref{example1}. Although adopting a smaller time step might improve accuracy, the constraint $\tau = \mathcal{O}(h^{2})$ would drastically increase the computational cost as $h$ vanishes. Consequently, we maintain the current settings for our simulations. However, for highly deformed shapes, performance may degrade because the current analysis is confined to short-time behavior. To enhance accuracy, adaptive refinement, such as dynamic insertion and removal of control points, may be required during simulation.

\section{Conclusion}\setcounter{equation}{0}
In this paper, we proposed an unfitted finite element framework for PDE-constrained shape optimization based on the $H^1$ shape gradient flow. The method combines three key ingredients: a shape gradient flow formulation with the state, adjoint, velocity, and flow map equations; a cubic spline representation for boundary evolution; and a CutFEM discretization with ghost penalty stabilization on unfitted meshes. This approach avoids the remeshing procedure during optimization while preserving a clear variational structure for both theoretical analysis and numerical implementation.

From a theoretical point of view, we established finite error estimates for the flow map, state, adjoint, and velocity variables under suitable regularity assumptions. In particular, for $k\ge 2$, $\tau=\mathcal{O}(h^k)$, and $\eta=\mathcal{O}(h^{k+\frac{1}{2}})$, we proved the bound
\begin{equation}\nonumber
        \N{\Be_\phi^n}_{H^1(\Omega^0)} + \N{e_u^n}_{H^1(\Omega_h^n)} +\N{e_p^n}_{H^1(\Omega_h^n)} +\N{\Be_w^n}_{H^1(\Omega_h^n)} \leq  C(\tau + h^k ),
    \end{equation}
which quantifies the coupling effects of the temporal, spatial, and geometric discretization errors.

Numerical experiments confirmed the theoretical findings: the method is robust for different initial shapes, and the observed convergence rates approach the expected order behavior as the mesh is refined. The results also indicate that coarse meshes may exhibit pre-asymptotic effects, as the relationship between the error and mesh size is inherently nonlinear. Since the current analysis may not extend naturally to long term simulations or large deformation cases, future work will focus on adaptive strategies for control point redistribution and mesh refinement. Additionally, we plan to extend this framework to more complex PDE constraints and three dimensional shape optimization problems.

 \medskip

\end{document}